\documentclass{amsart}
\usepackage[english]{babel}
\usepackage{amsmath}
\usepackage{amsthm}
\usepackage{amssymb}
\usepackage[all]{xy}
\usepackage{xkeyval}
\usepackage{graphicx}
\usepackage{makecell}
\usepackage{latexsym}
\usepackage{exscale}
\usepackage{comment}

\usepackage{amscd}
\usepackage{ytableau}

\usepackage{hyperref}

\usepackage[all]{xy}

\usepackage{amsbsy}

\def\frak{\mathfrak}

\newtheorem{prop}[equation]{Proposition}
\newtheorem*{prop*}{Proposition}
\newtheorem{thm}[equation]{Theorem}
\newtheorem*{thm*}{Theorem}
\newtheorem{lem}[equation]{Lemma}
\newtheorem*{lem*}{Lemma}

\newtheorem*{kor*}{Corollary}

\numberwithin{equation}{section}

\newcommand{\frg}{\mathfrak{g}}

\newcommand{\frk}{\mathfrak{k}}

\newcommand{\frp}{\mathfrak{p}}

\newcommand{\frt}{\mathfrak{t}}

\newcommand{\frsu}{\mathfrak{su}}
\newcommand{\frso}{\mathfrak{so}}
\newcommand{\frsp}{\mathfrak{sp}}

\newcommand{\bbar}{\,|\,}

\def\bbC{\mathbb{C}}

\def\bbR{\mathbb{R}}
\def\bbZ{\mathbb{Z}}

\let\ccdot\cdot
\def\cdot{\hbox to 2.5pt{\hss$\ccdot$\hss}}

\newcommand{\p}{{\frak p}}

\newcommand{\eq}{\begin{equation}}
	\newcommand{\eeq}{\end{equation}}
\newcommand{\eqn}{\begin{equation*}}
	\newcommand{\bmul}{\begin{multline*}}
		\newcommand{\eemul}{\end{multline*}}
	\newcommand{\eeqn}{\end{equation*}}
\newcommand{\pf}{\begin{proof}}
	\newcommand{\epf}{\end{proof}}

\newcommand{\la}{\lambda}

\renewcommand{\phi}{\varphi}

\newcommand{\eps}{\varepsilon}

\let\ssize\scriptstyle
\newif\ifFIRST\newdimen\MAXright\MAXright0pt
\def\sdynkin{\bgroup\eightpoint\dynkin}
\def\endsdynkin{\enddynkin\egroup}
\def\dynkin{\bgroup\FIRSTtrue\hskip.5em\setbox1\hbox{$\diagup$}%
	\setbox2\hbox{$\diagdown$}%
	\setbox0\hbox to2\wd1{\hrulefill}%
	\setbox3\hbox{$\bullet$}%
	\setbox4\hbox{$\times$}%
	\setbox7\hbox{$\circ$}%       (L.K.)
	\def\whiteroot##1{\ifFIRST\setbox5\hbox{$##1$}\ifdim\wd5>1.3em%       (L.K.)
		\hskip-.5em\hskip.5\wd5\fi\fi\FIRSTfalse%                             (L.K.)
		\hskip-.25em\raise1.5\wd3\hbox to0pt{\hss\hskip.45em$%                (L.K.)
			\ssize##1$\hss}\copy7\hskip-.25em\setbox6\hbox{$##1$}%                (L.K.)
		\MAXright\wd6}%                                                       (L.K.)
	\def\root##1{\ifFIRST\setbox5\hbox{$##1$}\ifdim\wd5>1.3em%
		\hskip-.5em\hskip.5\wd5\fi\fi\FIRSTfalse%
		\hskip-.25em\raise1.5\wd3\hbox to0pt{\hss\hskip.45em$%
			\ssize##1$\hss}\copy3\hskip-.25em\setbox6\hbox{$##1$}%
		\MAXright\wd6}%
	\def\whitedroot##1{\ifFIRST\setbox5\hbox{$##1$}\ifdim\wd5>1.3em% (L.K.)
		\hskip-.5em\hskip.5\wd5\fi\fi\FIRSTfalse% (L.K.)
		\hskip-.25em\lower1.8\wd3\hbox to0pt{\hss\hskip.45em$%  (L.K.)
			\ssize##1$\hss}\copy7\hskip-.25em\setbox6\hbox{$##1$}% (L.K.)
		\MAXright\wd6}%
	\def\whiterroot##1{\hskip-.25em\copy7\hbox to0pt{\hskip.3em$\ssize##1$\hss}%
		\hskip-.25em\setbox6\hbox{\hskip.6em$##1##1$}%
		\MAXright\wd6}%
	\def\droot##1{\ifFIRST\setbox5\hbox{$##1$}\ifdim\wd5>1.3em%
		\hskip-.5em\hskip.5\wd5\fi\fi\FIRSTfalse%
		\hskip-.25em\lower1.8\wd3\hbox to0pt{\hss\hskip.45em$%
			\ssize##1$\hss}\copy3\hskip-.25em\setbox6\hbox{$##1$}%
		\MAXright\wd6}%
	\def\rroot##1{\hskip-.25em\copy3\hbox to0pt{\hskip.3em$\ssize##1$\hss}%
		\hskip-.25em\setbox6\hbox{\hskip.6em$##1##1$}%
		\MAXright\wd6}%
	\def\norroot##1{\hskip-.36em\copy4\hbox to0pt{\hskip.3em$\ssize##1$\hss}%
		\hskip-.48em\setbox6\hbox{\hskip.6em$##1##1$}%
		\MAXright\wd6}%
	\def\noroot##1{\ifFIRST\setbox5\hbox{$##1$}\ifdim\wd5>1.3em%
		\hskip-.5em\hskip.5\wd5\fi\fi\FIRSTfalse%
		\hskip-.36em\raise1.5\wd3\hbox to0pt{\hss\hskip.6em$%
			\ssize##1$\hss}\copy4\hskip-.38em\setbox6\hbox{$##1$}%
		\MAXright\wd6}%
	\def\nodroot##1{\ifFIRST\setbox5\hbox{$##1$}\ifdim\wd5>1.3em%
		\hskip-.5em\hskip.5\wd5\fi\fi\FIRSTfalse%
		\hskip-.36em\lower1.8\wd3\hbox to0pt{\hss\hskip.6em$%
			\ssize##1$\hss}\copy4\hskip-.38em\setbox6\hbox{$##1$}%
		\MAXright\wd6}%
	\def\nolink{\hskip\wd0}%      (L.K.)
	\def\link{\raise.22em\copy0}%
	\def\llink##1{\raise.32em\copy0\hskip-\wd0%
		\raise.12em\copy0\hskip-.5\wd0\hbox to0pt{\hss$##1$\hss}\hskip.5\wd0}%
	\def\lllink##1{\raise.22em\copy0\hskip-\wd0\raise.32em\copy0\hskip-\wd0%
		\raise.12em\copy0\hskip-.5\wd0\hbox to0pt{\hss$##1$\hss}\hskip.5\wd0}%
	\def\rootupright##1{\hbox to0pt{\raise.45em\copy1\hskip-.25em\raise1.3\ht1%
			\hbox{\copy3\hskip.3em$\ssize##1$}\hss}%
		\setbox6\hbox{\hskip.6em\copy1\copy1$##1##1$}%
		\ifdim\MAXright<\wd6\MAXright\wd6\fi}%
	\def\whiterootupright##1{\hbox to0pt{\raise.45em\copy1\hskip-.25em\raise1.3\ht1% (L.K.)
			\hbox{\copy7\hskip.3em$\ssize##1$}\hss}% (L.K.)
		\setbox6\hbox{\hskip.6em\copy1\copy1$##1##1$}% (L.K.)
		\ifdim\MAXright<\wd6\MAXright\wd6\fi}% (L.K.)
	\def\norootupright##1{\hbox to0pt{\raise.45em\copy1\hskip-.36em\raise1.3\ht1%
			\hbox{\copy4\hskip.3em$\ssize##1$}\hss}%
		\setbox6\hbox{\hskip.6em\copy1\copy1$##1##1$}%
		\ifdim\MAXright<\wd6\MAXright\wd6\fi}%
	\def\rootdownright##1{\hbox to0pt{\raise-.5em\copy2\hskip-.25em\raise-1.35\ht1%
			\hbox{\copy3\hskip.3em$\ssize##1$}\hss}\setbox6%
		\hbox{\hskip.6em\copy2\copy2$##1##1$}%
		\ifdim\MAXright<\wd6\MAXright\wd6\fi}%
	\def\whiterootdownright##1{\hbox to0pt{\raise-.5em\copy2\hskip-.25em\raise-1.35\ht1% (L.K.)
			\hbox{\copy7\hskip.3em$\ssize##1$}\hss}\setbox6% (L.K.)
		\hbox{\hskip.6em\copy2\copy2$##1##1$}% (L.K.)
		\ifdim\MAXright<\wd6\MAXright\wd6\fi}% (L.K.)
	\def\rootdown##1{\hbox to0pt{\hskip-.05em\vrule height.25em depth.65em%
			\hskip-.25em\raise-.95em\hbox{\copy3\hskip.3em$\ssize##1$}\hss}%
		\setbox6\hbox{$##1$}%
		\ifdim\MAXright<\wd6\MAXright\wd6\fi}%
	\def\whiterootdown##1{\hbox to0pt{\hskip-.05em\vrule height.25em depth.65em% (L.K.)
			\hskip-.25em\raise-.95em\hbox{\copy7\hskip.3em$\ssize##1$}\hss}% (L.K.)
		\setbox6\hbox{$##1$}% (L.K.)
		\ifdim\MAXright<\wd6\MAXright\wd6\fi}% (L.K.)
	\def\dots{\hskip.5em\cdots\hskip.5em}}%
\def\enddynkin{\ifdim\MAXright>1em\hskip.5\MAXright\else\hskip.5em\fi\egroup}%

\begin{document} 
	
	\title[Dirac inequality]{Dirac inequality for highest weight Harish-Chandra modules I}
	\author{Pavle Pand\v zi\'c}
	\address[Pand\v zi\'c]{Department of Mathematics, Faculty of Science, University of Zagreb, Bijeni\v cka 30, 10000 Zagreb, Croatia}
	\email{pandzic@math.hr}
	\author{Ana Prli\'c}
	\address[Prli\'c]{Department of Mathematics, Faculty of Science, University of Zagreb, Bijeni\v cka 30, 10000 Zagreb, Croatia}
	\email{anaprlic@math.hr}
	\author{Vladim\'{\i}r Sou\v cek}
	\address[Sou\v cek]{Matematick\'y \'ustav UK, Sokolovsk\'a 83, 186 75 Praha 8, Czech Republic}
	\email{soucek@karlin.mff.cuni.cz}
	\author{V\'it Tu\v cek}
	\address[Tu\v cek]{Department of Mathematics, Faculty of Science, University of Zagreb, Bijeni\v cka 30, 10000 Zagreb, Croatia}
	\email{}
	\date{}
	\thanks{P.~Pand\v zi\'c, A.~Prli\'c are and V.~Tu\v cek were supported by the QuantiXLie  Center of Excellence, a project 
		cofinanced by the Croatian Government and European Union through the European Regional Development Fund - the Competitiveness and Cohesion Operational Programme 
		(KK.01.1.1.01.0004). V.~Sou\v cek is
		supported by the grant GACR GX19-28628.}
	\begin{abstract} 
		Let $G$ be a connected simply connected noncompact classical simple  Lie group of Hermitian type. Then $G$ 
		has unitary highest weight representations. The proof of the classification of unitary highest weight representations of $G$  given by Enright, Howe and Wallach is based on the Dirac inequality of Parthasarathy, Jantzen’s formula and Howe’s theory of dual pairs where one group in the pair is compact.
		In this paper we focus on the Dirac inequality which can be used to prove the classification in a more direct way.
	\end{abstract}

	\maketitle
	
	\section{Introduction}
	
	\bigskip
	
	In this introduction, we give an outline of the representation theory context for the results in this paper.
	However, in each of the cases described in Tables \ref{tab:table1} and \ref{tab: table2}, all the notions become completely explicit and elementary, and the representation theory context may be forgotten. Therefore, the reader who is not familiar with, or interested in, representation theory can mostly ignore the rest of the introduction and only check the concrete definitions given in Tables \ref{tab:table1} and \ref{tab: table2} before going to Sections 2 and 3 which contain our main results.

	Let $G$ be a connected simply connected noncompact classical simple  Lie group of Hermitian type.  (Exceptional Lie groups of Hermitian type are treated in \cite{PPST1}). Let $\Theta$ be a Cartan involution of $G$ and let $K$ be the group of fixed points of $\Theta$. If $Z$ denotes the center of $G$, then $K/Z$ is a maximal compact subgroup of $G/Z$.
	Let $\mathfrak{g}_0$ and $\mathfrak{k}_0$ be the Lie algebras of $G$ and $K$, respectively, and let $\mathfrak{g}_0 = \mathfrak{k}_0 \oplus \mathfrak{p}_0$ be the Cartan decomposition. Let $\mathfrak{t}_0$ be the common Cartan subalgebra of $\mathfrak{g}_0$ and $\mathfrak{k}_0$ and let $\mathfrak{g}$, $\mathfrak{k}$ and $\mathfrak{t}$ be the complexifications of the Lie algebras $\mathfrak{g}_0$, $\mathfrak{k}_0$ and $\mathfrak{t}_0$.  Let $\Delta^{+}_{\frg}\supset\Delta^{+}_{\frk}$ denote fixed sets of positive respectively positive compact roots. Since we assume that pair $(G, K)$ is Hermitian, we have a $K$--invariant decomposition $\frp = \frp^{+} \oplus \frp^{-}$ and $\frp^ {\pm}$ are abelian subalgebras of $\frp$. Let $\rho$ denote the half sum of positive roots  for $\frg$.
	
	A unitary representation of $G$ such that the underlying $(\frg, K)$--module is an irreducible quotient of a Verma module is called a unitary highest weight module. It is generated by a weight vector  that is annihilated by the action of all positive root spaces in $\frg$.
	
	For $\lambda \in \mathfrak{t}^{*}$ which are $\Delta^{+}_{\frk}$-- dominant integral (that means $\frac{2 \left < \lambda, \alpha \right >}{\left < \alpha, \alpha \right >}$ has to be an integer greater or equal to zero), let $N(\lambda)$ denote the generalized Verma module. By definition $N(\lambda)= S(\frp^{-}) \otimes F_{\lambda}$, where  $F_{\lambda}$ is the irreducible $\frk$--module with highest weight $\lambda$. The generalized Verma module  $N(\lambda)$ is a highest weight module ($\lambda$ is the highest weight of the $K$--type $F_{\lambda}$ but also a $\frg$--highest weight of $N(\lambda)$) which doesn't have to be irreducible or unitary. Our main goal in this paper is to determine those $N(\lambda)$ which correspond to unitary irreducible representation of $G$. We consider only real highest weights $\lambda$ since this is a necessary condition for unitarity. In case $N(\lambda)$ is not irreducible, we will consider the irreducible quotient $L(\lambda)$ of $N(\lambda)$ and we will determine those weights $\lambda$ which correspond to unitarizable $L(\lambda)$.
	
	Harish-Chandra has shown that $G$ admits non-trivial unitary highest weight modules precisely when $(G, K)$ is a Hermitian symmetric pair and that is precisely when the Lie algebra $\frg_0$ is one of the Lie algebras listed in tables \ref{tab:table1} and \ref{tab: table2}. To learn more about highest weight modules see \cite{A1}, \cite{DES}, \cite{EHW}, \cite{EJ}, \cite {ES}, \cite{J}. 
	\bigskip
	
	In \cite{EHW} (and independently in \cite{J}), a complete classification of the unitary highest weight modules was given using the Dirac inequality, Jantzen's fomula and Howe's theory of dual pairs.  They proved that  $L(\lambda)$ is unitarizable if and only if the strict Dirac inequality holds for all $K$--types occurring in $L(\lambda)$. This criterion is useful, but it is not easy to use because it is difficult to determine the $K$-types of $L(\lambda)$. The purpose of this and our future work is to show that the same result can be proved more directly using the Dirac inequality in a more substantial way. 
	
	\bigskip
	
	The structure of $S(\frp^{-})$ is very well known (see \cite{S}). The $K$--types of $S(\frp^ {-})$ are called the Schmid modules. For each of the  Lie algebras in Table \ref{tab: table2}, the general Schmid module $s$ is a nonnegative integer combination of the so called basic Schmid modules. The basic Schmid modules for each classical Lie algebra $\frg_0$ for which $(G, K)$ is a Hermitian symmetric pair are given in Table \ref{tab: table2}. 
	
	Let $U(\frg)$ be the universal enveloping algebra of $\frg$ and let $C(\mathfrak{p})$ be the Clifford algebra of $\mathfrak{p}$. The Dirac operator is an element of $U(\frg) \otimes C(\mathfrak{p})$ defined as $D = \sum_i b_i \otimes d_i$ where $b_i$ is a basis of $\mathfrak{p}$ and $d_i$ is the dual basis of $\mathfrak{p}$ with respect to the Killing form $B$. It is easy to show that $D$ is independent of the choice of $b_i$ and that it is $K$--invariant for the adjoint action on both factors. The Dirac operator acts on the tensor product $X \otimes S$ where $X$ is a $(\frg, K)$--module, and $S$ is the spin module for $C(\p)$. The square of the Dirac operator is very simple: 
	\[
	D^2 =  -(\text{Cas}_{\mathfrak{g}} \otimes 1 + \|\rho\|^2) + (\text{Cas}_{\mathfrak{k}_{\Delta}} + \|\rho_{\frk}^2\|),
	\]
	where $\rho_\frk$ is a half sum of the compact positive roots. There are many applications of the Dirac operators in representation theory (see \cite{D}, \cite{DH}, \cite{H}, \cite{HP1}, \cite{HP2}, \cite{HKP}).
	
	\bigskip
	
	The Dirac inequality is a very useful necessary condition for unitarity. More precisely, if a $(\frg, K)$--module is unitary, than $D$ is a self adjoint with respect to an inner product, so $D^2 \geq 0$. By the formula for $D^2$ the Dirac inequality becomes explicit on any $K$--type $F_{\tau}$ of $L(\lambda) \otimes S$ 
	$$
	\|\tau + \rho_{\frk} \|^2 \geq \|\lambda + \rho\|^2.
	$$
	
	In \cite{EHW} it was proved that $L(\lambda)$ is unitary if and only if $D^2 > 0$ on $F_\mu\otimes\bigwedge^{\rm top}\frp^+$ for any $K$--type $F_{\mu}$ of $L(\lambda)$ other than $F_{\lambda}$, that is if and only if 
	$$
	\| \mu + \rho \|^2 > \| \lambda + \rho \|^2.
	$$
	As we already said, it is difficult to determine the $K$--types of $L(\lambda)$. 
	
	The results of this paper provide examples for the following theorem:
	
	\begin{thm}
		\label{cor unit nonunit}
		Let us assume that $\frg, \rho, \lambda, s$ are as in  tables \ref{tab:table1} and \ref{tab: table2}.
		(1) Let $s_0$ be a Schmid module such that 
		the strict Dirac inequality
		\eq
		\label{strict di s}
		\|(\la-s)^++\rho\|^2> \|\la+\rho\|^2
		\eeq
		holds for any Schmid module $s$ of strictly lower level than $s_0$, and such that
		\[
		\|(\la-s_0)^++\rho\|^2< \|\la+\rho\|^2.
		\]
		Then $L(\la)$ is not unitary.
		
		(2)	If 
		\eq
		\label{strict di s}
		\|(\la-s)^++\rho\|^2> \|\la+\rho\|^2
		\eeq
		holds for all Schmid modules $s$, then $N(\la)$ is irreducible and unitary.
	\end{thm}
	
	In Theorem \ref{cor unit nonunit}, $(\lambda - s)^{+}$ is the unique $\frk$-dominant $W_{\frk}$-conjugate of $\lambda - s$, which means $(\lambda - s)^{+}$ is as in the third column of Table \ref{tab: table2}.
	
	The proof of the above theorem requires some tools from representation theory, so we will omit it in this paper and prove it in \cite{PPST2}.
	
	Another possible reason to provide this detailed study of the Dirac inequality is that \cite{EHW} prove a relationship between the norms of certain $K$-types in which the eigenvalue of the Dirac operator appears. Together with results of the current paper this could be potentially used to study convergence of $K$-type decompositions / series in the Hilbert spaces involved.
	
	\bigskip
	
	\begin{table}[h]
		\begin{center}
			\caption{$\rho$ and $W_{\frk}$}
			\label{tab:table1}
			\begin{tabular}{ |c|c|c|}
				\hline
				Lie algebra & $\rho$ & generators of $W_{\mathfrak{k}}$  \\ 
				\hline \hline
				$\mathfrak{sp}(2n, \mathbb{R})$ & $(n, n-1, \ldots, 2, 1)$ & $s_{\eps_i - \eps_j}, \ 1 \leq i < j \leq n$ \\ 
				\hline
				$\mathfrak{so}^{*}(2n)$ &  $(n-1, n-2, \ldots, 1, 0)$ & $s_{\eps_i - \eps_j}, \ 1 \leq i < j \leq n$ \\ 
				\hline
				$\mathfrak{su}(p, q) \ p \leq q$,  & $\left ( \frac{n - 1}{2}, \frac{n - 3}{2}, \ldots, \frac{- n + 3}{2}, \frac{-n + 1}{2} \right )$ & \makecell{$s_{\eps_i - \eps_j}, \ 1 \leq i < j \leq p$  \\ or $p + 1 \leq i < j \leq  n$} \\
				\hline
				$\mathfrak{so}(2, 2n - 2)$ & $(n-1, n-2, \ldots, 1, 0)$ & $s_{\eps_i \pm \eps_j}, \ 2 \leq i < j \leq n$
				\\
				\hline
				$\mathfrak{so}(2, 2n - 1)$ & $(n-\frac{1}{2}, n-\frac{3}{2}, \ldots, \frac{1}{2})$ & \makecell{$s_{\eps_i \pm \eps_j}, \ 2 \leq i < j \leq n$, \\ $s_{\eps_i}, \ 2 \leq i \leq n$}  \\
				\hline
			\end{tabular}
		\end{center}
	\end{table}
	
	\bigskip

	\begin{table}[h]
		\begin{center}
			\caption{The weights of basic Schmid modules and the condition for the $\frk$-highest weights $\lambda = (\lambda_1, \lambda_2, \ldots, \lambda_n)$}
			\label{tab: table2}
			\begin{tabular}{ |c|c|c| }
				\hline
				Lie algebra  & basic Schmid modules & highest weights \\ 
				\hline \hline
				$\mathfrak{sp}(2n, \mathbb{R})$  & \makecell{$
					s_i=(\underbrace{2,\dots,2}_i,0,\dots,0),$ \\ $i=1,\dots,n
					$} & \makecell{$\la_1\geq\la_2\geq\dots\geq\la_n,$ \\ $\la_i-\la_j\in\bbZ,\ 1\leq i,j\leq n.$.}\\ 
				\hline
				$\mathfrak{so}^{*}(2n)$  & \makecell{$s_i=(\underbrace{1,\dots,1}_{2i},0,\dots,0),$\\ $i=1,\dots, [n/2]$} & \makecell{$\la_1\geq\la_2\geq\dots\geq\la_n,$ \\ $\la_i-\la_j\in\bbZ,\ 1\leq i,j\leq n.$.} \\ 
				\hline
				$\mathfrak{su}(p, q) \ p \leq q$,   & \makecell{$s_i=(\underbrace{1,\dots,1}_i,0,\dots,0\bbar 0,\dots,0,\underbrace{-1,\dots,-1}_i),$ \\ $i=1,\dots,p$} & \makecell{
					$\la_1\geq\dots\geq\la_p;\ \la_{p+1}\geq\dots\geq\la_n,$ \\
					$\la_i-\la_j \in \mathbb{Z}, \ 1 \leq i < j \leq p$ \\ or $p + 1 \leq i < j \leq n$.}\\
				\hline
				$\mathfrak{so}(2, 2n - 2)$  & \makecell {
					$s_1=(1, 1, 0, \ldots, 0)$, \\
					$s_2 = (2, 0, 0, \ldots, 0)$} & \makecell{$\la_2\geq\la_3\geq\dots\geq\la_{n-1}\geq|\la_n|,$ \\ $\la_i-\la_j\in\bbZ,\ 2\leq i,j\leq n.$.}
				\\
				\hline
				$\mathfrak{so}(2, 2n - 1)$  & \makecell{
					$s_1 = (1, 1, 0, \ldots 0)$, \\
					$s_2 = (2, 0, 0, \ldots 0)$} & \makecell{$\la_2\geq\la_3\geq\dots\geq\la_n\geq 0,$ \\ $\la_i-\la_j\in\bbZ$ and $2 \lambda_i \in \mathbb{Z}$, \ $2\leq i,j\leq n.$ } \\
				\hline
			\end{tabular}
		\end{center}
	\end{table}

	In Table \ref{tab:table1}, $s_{\alpha}(\lambda) = \lambda - \frac{2 \langle \lambda, \alpha \rangle}{\langle\alpha,\alpha\rangle} \alpha$ is the reflection of $\lambda$ with respect to the hyperplane orthogonal to a root $\alpha$, $W_{\frk}$ is the Weyl group of $\frk$ generated by the $s_{\alpha}$ and $\mathbb{N}_0 = \mathbb{N} \cup \{0\}$.
	
	Here $\lambda$ and $\rho$ are elements of $\frt^*$ which is identified with $\bbC^n$, and $\eps_i$ denotes the projection to the $i$-th coordinate. The roots are certain functionals on $\frt^*$ and the relevant ones are those in the subscripts of the reflections $s$ in Table \ref{tab:table1}, like $\eps_i-\eps_j$ or $\eps_i+\eps_j$.
	
	\bigskip

	The rest of this paper is devoted to analyzing the Dirac inequality \eqref{strict di s} for various choices of $s$. Our analysis will be case by case for the Lie algebras as in the above tables but first we prove two auxiliary results that are going to help us in each of the cases. 
	
	\newpage

	\section{Some technical lemmas}

	\begin{lem}
		\label{gen prv}
		Let $\frg$ be one of the Lie algebras listed in the above tables. Let $\mu$ and $\nu$ be weights as in Table \ref{tab: table2}. Let $w_1,w_2\in W_\frk$. Then
		\[
		\|(w_1\mu-w_2\nu)^++\rho\|^2\geq \|(\mu-\nu)^++\rho\|^2.
		\]
	\end{lem}
	In Lemma \ref{gen prv}, $(w_1\mu-w_2\nu)^+$ is the unique dominant $W_{\frk}$-conjugate of $w_1\mu-w_2\nu$, which means $(w_1\mu-w_2\nu)^+$ is as in the third column of  Table \ref{tab: table2}. The proof requires some representation theory and we leave it for \cite{PPST2}.
	
	\medskip
	
	In the computations needed to prove the Dirac inequality, we will repeatedly use the following elementary lemma.
	
	\begin{lem}
		\label{red sp}
		\begin{comment}
		Let $(\frg,\frk)$ be any of the Hermitian symmetric pairs for a simple complex Lie algebra $\frg$. Using the standard positive roots, let $\rho$ be the half sum of positive $\frg$-roots. Let $n$ denote the rank of $\frg$ or $\frk$.
		\end{comment}
		Let $\mu$ and $\nu$ be two $n$-tuples with strictly decreasing coordinates and let $\rho$ be as in Table \ref{tab:table1}.
		
		(1) Suppose there are $u,v$, $1\leq u<v\leq n$, such that
		\[
		\mu_u=\nu_u,\dots,\mu_{v-1}=\nu_{v-1};\ \mu_v<\nu_v.
		\]
		Let $\mu',\nu'$ be obtained from $\mu,\nu$ by moving the $v$-th coordinate to the $u$-th place and shifting the coordinates in between to the right, i.e., 
		\[
		\mu'=(\mu_1,\dots,\mu_{u-1},\mu_v,\mu_u,\dots,\mu_{v-1},\mu_{v+1},\dots,\mu_n)
		\]
		and likewise for $\nu'$. Then
		\eq
		\label{din11}
		\|\mu+\rho\|^2-\|\nu+\rho\|^2 > \|\mu'+\rho\|^2-\|\nu'+\rho\|^2.
		\eeq
		(2) Suppose there are $u,v$, $1\leq u<v\leq n$, such that
		\[
		\mu_u>\nu_u;\ \mu_{u+1}=\nu_{u+1},\dots,\mu_v=\nu_v.
		\]
		Let $\mu',\nu'$ be obtained from $\mu,\nu$ by moving the $u$-th coordinate to the $v$-th place and shifting the coordinates in between to the left, i.e., 
		\[
		\mu'=(\mu_1,\dots,\mu_{u-1},\mu_{u+1},\dots,\mu_v,\mu_{u},\mu_{v+1},\dots,\mu_n)
		\]
		and likewise for $\nu'$. Then
		\eq
		\label{din11a}
		\|\mu+\rho\|^2-\|\nu+\rho\|^2 > \|\mu'+\rho\|^2-\|\nu'+\rho\|^2.
		\eeq
		(3) Let $\mu$ be a $n$-tuple, such that for some $s,t\geq 1$ and for some $u$ between 1 and $n-s-t$, 
		\[
		\mu=(\mu_1,\dots,\mu_u,\underbrace{x+1}_s,\underbrace{x}_t,\mu_{u+s+t+1},\dots,\mu_n).
		\]
		Let
		\[
		\mu'=(\mu_1,\dots,\mu_u,\underbrace{x}_t,\underbrace{x+1}_s,\mu_{u+s+t+1},\dots,\mu_n).
		\]
		Then 
		\[
		\|\mu+\rho\|^2>\|\mu'+\rho\|^2
		\]
	\end{lem} 
	\pf
	(1) The difference of the two sides of \eqref{din11} is
	\begin{multline*}
		[(\mu_v+\rho_v)^2-(\nu_v+\rho_v)^2]-[(\mu_v+\rho_u)^2-(\nu_v+\rho_u)^2]\\
		=(\mu_v-\nu_v)(\mu_v+\nu_v+2\rho_v)-(\mu_v-\nu_v)(\mu_v+\nu_v+2\rho_u)\\
		=(\mu_v-\nu_v)(2\rho_v-2\rho_u)=2(\nu_v-\mu_v)(\rho_u-\rho_v).
	\end{multline*} 
	Since $\nu_v>\mu_v$ by assumption, and since $\rho_u>\rho_v$, the claim follows.
	
	\medskip
	
	The proof of (2) is analogous to the proof of (1).
	
	\medskip
	
	(3) We prove the required inequality by successively switching pairs $x+1,x$ as follows. Let
	\[
	\mu''=(\mu_1,\dots,\mu_u,x,\underbrace{x+1}_{s-1},\underbrace{x}_{t-1},x+1,\mu_{u+s+t+1},\dots,\mu_n).
	\]
	Then we see, factoring the differences of squares, that
	\begin{multline*}
		\|\mu+\rho\|^2-\|\mu''+\rho\|^2\\
		=(x+1+\rho_{u+1})^2-(x+\rho_{u+1})^2+(x+\rho_{u+s+t})^2-(x+1+\rho_{u+s+t})^2\\
		=(2x+1+2\rho_{u+1})-(2x+1+2\rho_{u+s+t})=2(\rho_{u+1}-\rho_{u+s+t})>0.
	\end{multline*}
	Now we continue with the next pair until we reach $\mu'$. The claim follows.
	\epf

	\section{Dirac inequalities}
	
	\subsection{Dirac inequality for $\frsp(2n,\bbR)$}
	\begin{comment}
	Let $\lambda$ be the highest weight of an irreducible $(\frg,K)$-module $L(\lambda)$. Recall that $L(\lambda)$ is the simple quotient of the generalized Verma module $N(\lambda)$ with highest weight $\la$. The infinitesimal character of $L(\la)$ and $N(\la)$ is represented by the parameter $\lambda+\rho$, which must be dominant regular integral for $\frk$. In other words, the coordinates of $\la+\rho$ must be strictly decreasing and their differences must be integers.
	\end{comment} 
	The basic Schmid $\frk$-submodules of $S(\frp^-)$ have lowest weights $-s_i$, where
	\[
	s_i=(\underbrace{2,\dots,2}_i,0,\dots,0),\qquad i=1,\dots,n.
	\]
	The highest weight $(\mathfrak{g}, K)$--modules have highest weight of the form $\lambda = (\lambda_1, \lambda_2, \ldots, \lambda_n)$, where $\lambda_i - \lambda_j \in \mathbb{N}_{0}, \ i > j$.

	In this case $\rho = (n, n-1, \ldots, 2, 1).$ The basic necessary condition for unitarity is the Dirac inequality
	\eq
	\label{basic di sp}
	\|(\la-s_1)^++\rho\|^2\geq\|\la+\rho\|^2.
	\eeq
	
	To understand this inequality better, let $q\leq r$ be integers in $[1,n]$ such that
	\eq
	\label{lambda sp}
	\lambda=(\underbrace{\lambda_1,\dots,\lambda_1}_q,\underbrace{\lambda_1-1,\dots,\lambda_1-1}_{r-q},\lambda_{r+1},\dots,\lambda_n),
	\eeq
	with $\lambda_1-2\geq \lambda_{r+1}\geq\dots\geq\lambda_n$.
	Then 
	\[
	(\la-s_1)^+=(\underbrace{\lambda_1,\dots,\lambda_1}_{q-1},\underbrace{\lambda_1-1,\dots,\lambda_1-1}_{r-q},\la_1-2,\lambda_{r+1},\dots,\lambda_n)=\la-(\eps_q+\eps_r).
	\]
	\begin{comment}
	The root $\gamma=\eps_q+\eps_r$, or the integers $q$ and $r$, contain the same information about $\la$ as the root systems $Q$ and $R$ of \cite{EHW}. In fact, it is easy to see that $Q$ is the root system of the subalgebra $\frsp(2q,\bbR)$ of $\frg$ built on the first $q$ coordinates, while $R$ is the root system of the subalgebra $\frsp(2r,\bbR)$ of $\frg$ built on the first $r$ coordinates.
	\end{comment}
	The inequality \eqref{basic di sp} now becomes equivalent to
	$\|\la+\rho-\gamma\|^2\geq\|\la+\rho\|^2$, or to
	\[
	2\langle \la+\rho,\eps_q+\eps_r\rangle \leq \|\gamma\|^2,
	\]
	where $\gamma=\eps_q+\eps_r$.
	If $q\neq r$, then $\la_q=\la_1$, $\la_r=\la_1-1$ and $\|\gamma\|^2=2$, and the inequality becomes
	\eq
	\label{basic cond sp}
	\la_1\leq -n+\frac{r+q}{2}.
	\eeq
	If $q=r$, then $\la_q=\la_r=\la_1$ and $\|\gamma\|^2=4$, and the inequality is again \eqref{basic cond sp}.
	
	\begin{comment}
	Recall that by a result of Schmid \cite{S}, the $K$-module  
	$S(\frp^-)$ is multiplicity free, and its $K$-types are the irreducible $K$-modules $F_{-s}$ with lowest weights
	$-s$, where 
	\eq
	\label{gen schmid}
	s=(2b_1,\dots,2b_n),\qquad b_i\in\bbZ,\quad b_1\geq\dots\geq b_n\geq 0.
	\eeq
	We will refer to such $F_{-s}$, or to $s$, as Schmid modules; they are nonnegative integer combinations of the basic Schmid modules
	\[
	s_i=(\underbrace{2,\dots,2}_i,0,\dots,0),\qquad i=1,\dots,n.
	\] 
	Since $N(\lambda)=F_\lambda\otimes S(\frp^-)$, the $K$-types of $N(\lambda)$ are irreducible summands of various $F_\lambda\otimes F_{-s}$, where $F_{-s}$ is as above. We want to see for which of them the Dirac inequality holds. Note that by Proposition \ref{prv}, if for some $s$ the PRV component of $F_\lambda\otimes F_{-s}$ satisfies the Dirac inequality, i.e., if
	\eq
	\label{din1}
	\|(\la-s)^++\rho\|^2\geq \|\la+\rho\|^2,
	\eeq
	then the same is true for all components of $F_\lambda\otimes F_{-s}$. Thus it will be enough for us to work with the PRV components. We start by examining what happens for the first $q$ basic Schmid modules $s=s_1,\dots,s_q$. Since $s_1=\beta$ (the highest root), we already know that the Dirac inequality for $s_1$ is satisfied if and only if \eqref{basic cond sp} holds. 
	\end{comment}
	Now we are going to see in which cases the Dirac inequality holds for $s_i, \ i \in \{ 2, \ldots, n \}$.
	We have
	\begin{gather*}
		(\lambda-s_i)^+ =(\underbrace{\lambda_1}_{q-i},\underbrace{\lambda_1-1}_{i},\underbrace{\lambda_1-1}_{r-q-i},\underbrace{\lambda_1-2}_{i},\lambda_{r+1},\dots,\lambda_n)\\
		\lambda =(\underbrace{\lambda_1}_{q-i},\underbrace{\lambda_1}_{i},\underbrace{\lambda_1-1}_{r-q-i},\underbrace{\lambda_1-1}_{i},\lambda_{r+1},\dots,\lambda_n),
	\end{gather*}
	so 
	\eq
	\label{din1}
	\|(\la-s_i)^++\rho\|^2\geq \|\la+\rho\|^2,
	\eeq
	is equivalent to non-negativity of
	\begin{multline*}
		\sum_{u=q-i+1}^q\left[(\lambda_1-1+\rho_u)^2-(\lambda_1+\rho_u)^2\right]+\sum_{v=r-i+1}^r \left[(\lambda_1-2+\rho_v)^2-(\lambda_1-1+\rho_v)^2\right],
	\end{multline*}
	where $\rho_u = n - u + 1$, $\rho_v = n - v + 1$.
	Factoring the differences of squares this expression becomes
	\begin{multline*}
		-\sum_{u=q-i+1}^q(2\lambda_1-1+2\rho_u)-\sum_{v=r-i+1}^r(2\lambda_1-3+2\rho_v)= \\
		-i(2\lambda_1-1)-2[(n-q+i)+\dots +(n-q+1)]-i(2\lambda_1-3)-2[(n-r+i)+\dots +(n-r+1)]= \\
		-4i\lambda_1+4i-2i(n-q)-2i(n-r)-4(i+\dots+1)= \\
		-2i(2\lambda_1+2n -2-q-r)-2i(i+1).
	\end{multline*}
	Dividing by $(-2i)$, we see that \eqref{din1} is equivalent to
	\[
	2\lambda_1 +2n -2-q-r+i+1\leq 0,
	\]
	or 
	\eq
	\label{din2}
	\lambda_1\leq -n+\frac{r+q-i+1}{2}.
	\eeq
	Moreover, it is clear from the above argument that \eqref{din1} holds strictly if and only if \eqref{din2} holds strictly.
	Note also that for $i=1$, \eqref{din2} is exactly our basic inequality \eqref{basic cond sp}, while for $i=q$ we get 
	\[
	\lambda_1\leq -n+\frac{r+1}{2}.
	\]
	
	\begin{thm}
		\label{lem unit nonunit sp} Let $\la$ be as in \eqref{lambda sp}. Then:
		\begin{enumerate}
			\item
			If for some integer $i\in[1,q]$ 
			\[
			\lambda_1< -n+\frac{r+q-i+1}{2},
			\]
			then the Dirac inequality holds strictly for any Schmid module $s = (2b_1, \ldots, 2b_n), \ b_j \in \mathbb{Z}, \ b_1 \geq \cdots \geq b_n \geq 0$ with at most $i$ nonzero components, i.e.
			\begin{equation}\label{smt}
				\| (\lambda - s)^{+} + \rho \|^2 > \| \lambda + \rho \|^2.
			\end{equation}
			\item
			If 
			$$
			\lambda_1 < -n + \frac{r + 1}{2},
			$$
			then the Dirac inequality holds strictly for any Schmid module $s$.
		\end{enumerate}
	\end{thm}
	\pf 
	Let 
	\eq
	\label{schmid i}
	s=(2b_1,\dots,2b_j,0,\dots,0),\qquad j\leq i.
	\eeq
	If $s=s_j$ is a basic Schmid module, then we have already seen
	that strict \eqref{din1} holds for $s$, because strict \eqref{din2} holds by assumption. So let us 
	assume that $s$ is not a basic Schmid module, i.e.,  $b_1\geq 2$. We will show how to reduce the claim to the same claim for $s'$ with smaller $b_1$ and then use induction.
	
	Let $k$, $1\leq k\leq j$, be such that $b_1=\dots=b_k>b_{k+1}$; so
	\[
	s=(\underbrace{2b_1,\dots,2b_1}_k,2b_{k+1}\dots,2b_j,0,\dots,0).
	\] 
	Let
	\[
	s'=s-s_k=(\underbrace{2b_1-2,\dots,2b_1-2}_k,2b_{k+1}\dots,2b_j,0,\dots,0).
	\]
	We claim that 
	\eq
	\label{din3}
	\|(\lambda-s)^++\rho\|^2\geq\|(\lambda-s')^++\rho\|^2.
	\eeq
	If we prove this, then we can do induction on $b_1$ and conclude that \eqref{din1} holds strictly for all $s$ as in \eqref{schmid i} ($b_1=1$ corresponds to $s=s_j$, the case already handled).
	
	To prove \eqref{din3}, we first note that 
	\begin{gather*}
		(\lambda-s)^+=(\underbrace{\lambda_1}_{q-j},\underbrace{\lambda_1-1}_{r-q},\lambda_1-2b_j,\dots,\lambda_1-2b_{k+1},\underbrace{\lambda_1-2b_1}_k,\lambda_{r+1},\dots,\lambda_n)^+; \\
		(\lambda-s')^+=(\underbrace{\lambda_1}_{q-j},\underbrace{\lambda_1-1}_{r-q},\lambda_1-2b_j,\dots,\lambda_1-2b_{k+1},\underbrace{\lambda_1-2b_1+2}_k,\lambda_{r+1},\dots,\lambda_n)^+,
	\end{gather*}
	where the first $r$ coordinates in each expression are already arranged in descending order, and $\lambda_{r+1},\dots,\lambda_n$ have to be put into proper places.
	
	By Lemma \ref{red sp}, to prove \eqref{din3} it is enough to prove
	\eq
	\label{din10}
	\|\eta+\rho\|^2\geq\|\eta'+\rho\|^2,
	\eeq
	where 
	\begin{gather*}
		\eta=(\underbrace{\lambda_1}_{q-j},\underbrace{\lambda_1-1}_{r-q},\lambda_1-2b_j,\dots,\lambda_1-2b_{k+1},\underbrace{\lambda_1-2b_1}_k,\lambda_{r+1},\dots,\lambda_n); \\
		\eta'=(\underbrace{\lambda_1}_{q-j},\underbrace{\lambda_1-1}_{r-q},\lambda_1-2b_j,\dots,\lambda_1-2b_{k+1},\underbrace{\lambda_1-2b_1+2}_k,\lambda_{r+1},\dots,\lambda_n).
	\end{gather*}
	
	In more detail, to see that \eqref{din3} follows from \eqref{din10},
	we first use Lemma \ref{red sp}(3) to move the coordinates $\la_1-2b_1$ of $(\la-s)^+$ to the left past those of $\la_{r+1},\dots,\la_n$ that are equal to $\la_1-2b_1+1$. Now the coordinates $\la_1-2b_1$ of $(\la-s)^+$ are positioned exactly above the coordinates $\la_1-2b_1+2$ of $(\la-s')^+$, and we can use Lemma \ref{red sp}(1) to move these coordinates simultaneously to the left of those of $\la_{r+1},\dots,\la_n$ that are $\geq \la_1-2b_1+2$.
	
	To prove \eqref{din10}, we compute factoring the differences of squares,
	\begin{multline*}
		\|\eta+\rho\|^2-\|\eta'+\rho\|^2=\sum_{j=r-k+1}^r [(\la_1-2b_1+\rho_j)^2-(\la_1-2b_1+2+\rho_j)^2]\\
		=\sum_{j=r-k+1}^r (-2)(2\la_1-4b_1+2+2\rho_j)\\
		=-2k(2\la_1-4b_1+2)-4[(n-r+k)+\dots+(n-r+1)]\\
		=-2k(2\la_1-4b_1+2+2n-2r+k+1).
	\end{multline*}
	This expression is positive if and only if
	\eq
	\label{din12}
	2\la_1+2n-4b_1-2r+k+3<0.
	\eeq
	Since by assumption 
	\[
	\la_1<-n+\frac{r+q-i+1}{2},
	\]
	and since $b_1\geq 2$, we see that \eqref{din12} will follow if we prove
	\[
	r+q-i+1-8-2r+k+3<0,
	\]
	or
	\[
	-r+q-i+k-4<0.
	\]
	The last inequality is however obvious since $q\leq r$ and $k\leq i$.
	
	So we have proved \eqref{din10}, and as explained above, this finishes the proof of Theorem \ref{lem unit nonunit sp}(1).
	
	We now prove Theorem \ref{lem unit nonunit sp}(2). If we take $i=q$ in the already proved Theorem \ref{lem unit nonunit sp}(1), we see that \eqref{din1} holds strictly for all $s$ having at most $q$ nonzero components.
	
	Assume now that $q<r$, and that 
	\eq
	\label{schmid i2}
	s=(2b_1,\dots,2b_i,0,\dots,0),\qquad b_i\geq 1,\quad q<i\leq r.
	\eeq
	Let $s'=(s-s_q)^+$ and let $\la'=(\la-s_q)^+$, where $s_q$ is the $q$th basic Schmid module. We claim that
	\eq
	\label{red mid}
	\|(\lambda-s)^++\rho\|^2\geq\|(\lambda'-s')^++\rho\|^2
	\eeq
	and that
	\eq
	\label{red mid2}
	\|\lambda+\rho\|^2 < \|\lambda'+\rho\|^2.
	\eeq
	If we prove these two inequalities, then it follows that 
	\eq
	\label{din aux 1}
	\|(\lambda-s)^++\rho\|^2-\|\lambda+\rho\|^2 > \|(\lambda'-s')^++\rho\|^2-\|\lambda'+\rho\|^2,
	\eeq
	so the strict inequality \eqref{smt} for $\la$ and $s$ will follow if we can prove the strict \eqref{din1} for $\la'$ and $s'$. 
	
	Since 
	\[
	\la'=(\underbrace{\la_i-1}_{r-q},\underbrace{\la_i-2}_q,\la_{r+1},\dots,\la_n),
	\] 
	the analogue $r'$ of $r$ for $\la'$ satisfies $r'\geq r$.  Moreover, $$\la'_1=\la_1-1<\la_1 < -n + \frac{r + 1}{2} \leq  -n + \frac{r' + 1}{2}.$$
	
	Furthermore, the analogue $q'$ of $q$ for $\la'$ satisfies $q'=r-q< r\leq r'$, and the number of nonzero coordinates of $s'$ is $\leq i\leq r'$.
	Also, the sum of coordinates of $s'$ is smaller than the sum of coordinates of $s$, and as we keep reducing as above, the number of nonzero coordinates will also become smaller. So we can use induction and keep reducing $s$ until we come to the situation $i\leq q$, which we already handled. Thus to prove Theorem \ref{lem unit nonunit sp}(2) for $s$ as in \eqref{schmid i2}, it suffices to prove \eqref{red mid} and \eqref{red mid2}. 
	
	To prove \eqref{red mid}, we apply Lemma \ref{gen prv} as follows. Let 
	\[
	\mu=\la'=(\lambda-s_q)^+,\qquad \nu=s'=(s-s_q)^+.
	\]
	Let $w_1,w_2\in W_\frk$ be such that $\lambda-s_q=w_1\mu$ and $s-s_q=w_2\nu$. Then Lemma \ref{gen prv} implies
	\[
	\|(\la-s)^++\rho\|^2=\|(w_1\mu-w_2\nu)^++\rho\|^2\geq \|(\mu-\nu)^++\rho\|^2=\|(\la'-s')^++\rho\|^2,
	\]
	and this is exactly \eqref{red mid}.
	
	To prove \eqref{red mid2}, we compute
	\begin{multline*}
		\|\la+\rho\|^2-\|\la'+\rho\|^2=\\
		\sum_{j=1}^q(\la_1+\rho_j)^2+\sum_{j=q+1}^r(\la_1-1+\rho_j)^2-\sum_{j=1}^{r-q}(\la_1-1+\rho_j)^2-\sum_{j=r-q+1}^r(\la_1-2+\rho_j)^2\\
		=q\la_1^2+\sum_{j=1}^q(2\la_1\rho_j+\rho_j^2)+(r-q)(\la_1-1)^2+\sum_{j=q+1}^r[2(\la_1-1)\rho_j+\rho_j^2]\\
		-(r-q)(\la_1-1)^2-\sum_{j=1}^{r-q}[2(\la_1-1)\rho_j+\rho_j^2]-q(\la_1-2)^2-\sum_{j=r-q+1}^r[2(\la_1-2)\rho_j+\rho_j^2]\\
		=q[\la_1^2-(\la_1-2)^2]-2\sum_{j=q+1}^r\rho_j+2\sum_{j=1}^{r-q}\rho_j+4\sum_{j=r-q+1}^r\rho_j\\
		=q(4\la_1-4)-2\sum_{j=q+1}^r\rho_j+2\sum_{j=1}^r\rho_j+2\sum_{j=r-q+1}^r\rho_j\\
		=4q(\la_1-1)+2\sum_{j=1}^q\rho_j+2\sum_{j=r-q+1}^r\rho_j\\
		=4q(\la_1-1)+[2q(n-q)+2(q+\dots+1)]+[2q(n-r)+2(q+\dots+1)]\\
		=2q[(2\la_1-2)+(n-q)+(n-r)+(q+1)]=2q(2\la_1+2n-r-1).
	\end{multline*}
	Since $\lambda_1 < -n + \frac{r + 1}{2}$, the last expression is clearly $<0$, and this implies \eqref{red mid2}.
	So we have proved Theorem \ref{lem unit nonunit sp}(2) for $s$ as in \eqref{schmid i2}.
	
	Finally, suppose that $r<n$ and that 
	\eq
	\label{schmid i3}
	s=(2b_1,\dots,2b_i,0,\dots,0),\qquad b_i\geq 1,\quad i> r.
	\eeq
	
	Let $s'=s-2\eps_i$ and let $\la'=(\la-2\eps_i)^+$. We claim that
	\eq
	\label{red last}
	\|(\lambda-s)^++\rho\|^2\geq\|(\lambda'-s')^++\rho\|^2
	\eeq
	and that
	\eq
	\label{red last2}
	\|\lambda+\rho\|^2\leq\|\lambda'+\rho\|^2.
	\eeq
	These two equations imply that 
	\[
	\|(\lambda-s)^++\rho\|^2-\|\lambda+\rho\|^2\geq\|(\lambda'-s')^++\rho\|^2-\|\lambda'+\rho\|^2.
	\]
	Moreover, since $\la'$ has the same $r$ as $\la$, and also $\la'_1=\la_1$, we have $\la_1' < -n + \frac{r' + 1}{2}$. 
	Now we can use induction and keep decreasing the last nonzero coordinate of $s$ until we come to the situation $i\leq r$. In this case, we already proved the strict Dirac inequality. So it suffices to prove \eqref{red last} and \eqref{red last2}. 
	
	To prove \eqref{red last}, we use Lemma \ref{gen prv}.
	Namely, we set $\mu=\lambda'$ and $\nu=s'$, and choose $w\in W_\frk$ such that $w\mu=\lambda-2\eps_i$. Then Lemma \ref{gen prv} for $w_1=w$, $w_2=1$ implies
	\[
	\|(\lambda-2\eps_i-s')^++\rho\|^2\geq\|(\lambda'-s')^++\rho\|^2.
	\]
	This is equivalent to \eqref{red last} because clearly
	$\lambda-2\eps_i-s'=\lambda-s$.
	
	To prove \eqref{red last2}, we first note that 
	\[
	\la'=(\la-2\eps_i)^+=\la-\eps_j-\eps_k
	\]
	for some $j,k$ satisfying $i\leq j\leq k\leq n$. 
	So \eqref{red last2} becomes
	\[
	\|\lambda+\rho\|^2\leq\|\lambda+\rho-(\eps_j+\eps_k)\|^2=\|\lambda+\rho\|^2-2\langle\la+\rho,\eps_j+\eps_k\rangle+\|\eps_j+\eps_k\|^2,
	\]
	and this is equivalent to
	\[
	2\langle\la+\rho,\eps_j+\eps_k\rangle\leq\|\eps_j+\eps_k\|^2.
	\]
	We claim that in fact the left side of the last inequality is negative, i.e., that 
	\[
	\la_j+\la_k+\rho_j+\rho_k<0.
	\]
	Since $\la_j$ and $\la_k$ are both $\leq\la_1-2$, and since 
	$\rho_j$ and $\rho_k$ are both $\leq n-r$ (because $j,k\geq r$), it is enough to prove that
	\[
	2\la_1-4+2n-2r< 0.
	\]
	By assumption, $\la$ is in the continuous part of its line, i.e., $2\la+2n<r+1$, so the last inequality is obvious.
	This finishes the proof of Theorem \ref{lem unit nonunit sp}(2).
	\epf
	
	\subsection{Dirac inequality for $\frso^*(2n)$}
	\label{subsec di so*}
	The basic Schmid $\frk$-submodules of $S(\frp^-)$ have lowest weights $-s_i$, where
	\[
	s_i=(\underbrace{1,\dots,1}_{2i},0,\dots,0),\qquad i=1,\dots, [n/2].
	\]
	Moreover, all irreducible $\frk$-submodules of $S(\frp^-)$ have lowest weights $-s$, where
	\eq
	\label{gen sch so}
	s=(b_1,b_1,b_2,b_2,\dots,b_j,b_j,0,\dots,0)
	\eeq
	for some $j$, $1\leq j\leq [n/2]$, and some positive integers $b_1\geq b_2\geq\dots\geq b_j$.
	
	The highest weight $(\frg,K)$-modules have highest weights of the form
	\[
	\la=(\la_1,\la_2,\dots,\la_n),\qquad \la_1\geq\la_2\geq\dots\geq\la_n,\qquad \la_i-\la_j\in\bbZ,\ 1\leq i,j\leq n.
	\]	
	
	In this case $\rho = (n-1, n-2, \ldots, 1, 0).$
	
	The basic necessary condition for unitarity is, as before, the Dirac inequality
	\eq
	\label{basic di so}
	\|(\la-s_1)^++\rho\|^2\geq\|\la+\rho\|^2.
	\eeq
	To make this inequality more precise, we as before write
	\[
	(\la-s_1)^+=\la-\gamma.
	\]
	Then \eqref{basic di so} becomes 
	\[
	\|\la-\gamma+\rho\|^2\geq\|\la+\rho\|^2,
	\]
	and since $\|\gamma\|^2=2$, this is equivalent to
	\[
	\langle \la+\rho,\gamma\rangle\leq 1.
	\]
	There are two basic cases:
	\smallskip
	
	\noindent{\bf Case $\mathbf{1}$: $\mathbf{\la_1>\la_2}$.} Let $q\in[2,n]$ be such that $\la_2=\dots=\la_q$ and, in case $q<n$, $\la_q>\la_{q+1}$. Then $\gamma=\eps_1+\eps_q$, and since $\la_q=\la_2$, the basic inequality becomes
	\eq 
	\label{basic di so c1}
	\la_1+\la_2\leq -2n+q+2.
	\eeq
	
	\noindent{\bf Case $\bold 2$: $\bold{\la_1=\la_2}$.} Let $p\in[2,n]$ be such that $\la_1=\la_2=\dots=\la_p$ and, in case $p<n$, $\la_p>\la_{p+1}$. Then $\gamma=\eps_{p-1}+\eps_p$, and since $\la_{p-1}=\la_p=\la_1$, the basic inequality becomes
	\eq 
	\label{basic di so c2}
	\la_1\leq -n+p.
	\eeq
	\smallskip

	\begin{thm}
		\label{lem so c1} Let $\la$ be as in Case 1, and suppose that \eqref{basic di so c1} holds strictly. Let $s\neq 0$ be as in \eqref{gen sch so}. Then
		\[
		\|(\la-s)^++\rho\|^2>\|\la+\rho\|^2.
		\]
	\end{thm}
	
	\begin{comment}
	The theorem follows from the lemma by similar arguments as we were using in the symplectic case. Namely, Lemma implies that the strict Dirac inequality holds for all $K$-types of the generalized Verma module $N(\lambda)$ tensored with the spin module (except $F_\la\otimes 1$). It now follows from ??? that $N(\la)$ is irreducible.
	(If $v\in N(\la)$ is a singular vector, then $D^2(v\otimes 1)=0$.) Unitarity of $L(\la)=N(\la)$ follows from the strict Dirac inequality by \cite{EHW}, Proposition 3.9. Moreover, if \eqref{basic di so c1} holds as an equality, then  $L(\la)$ is unitary by continuity.
	\smallskip
	\end{comment}
	
	To prove the theorem, assume first that $2j\leq q$. Let $k\leq j$ be the largest integer such that $b_1=\dots=b_k$. Let
	\[
	s'=(\underbrace{b_1-1,\dots,b_1-1}_{2k},b_{k+1},b_{k+1},\dots,b_j,b_j,0,\dots,0).
	\]
	We claim that 
	\[
	\|(\la-s)^++\rho\|^2>\|(\la-s')^++\rho\|^2.
	\]
	The claim implies the theorem for $2j\leq q$ by induction on $b_1+\dots+b_j$. 
	
	To prove the claim, we first note that $(\la-s)^+$ and $(\la-s')^+$ both contain coordinates
	\[
	\underbrace{\la_2,\dots,\la_2}_{q-2j},\la_2-b_j,\dots,\la_2-b_{k+1},\la_{q+1},\dots,\la_n.
	\] 
	The remaining coordinates of $(\la-s)^+$ are 
	\[
	\la_1-b_1,\underbrace{\la_2-b_1,\dots,\la_2-b_1}_{2k-1},
	\] 
	while the remaining coordinates of $(\la-s')^+$ are 
	\[
	\la_1-b_1+1,\underbrace{\la_2-b_1+1,\dots,\la_2-b_1+1}_{2k-1}.
	\] 
	Using Lemma \ref{red sp} (1), we can move the equal coordinates to the right. We use that to move $\la_{q+1},\dots,\la_n$ all the way to the right, and to move those of
	$\la_2,\dots,\la_2,\la_2-b_j,\dots,\la_2-b_{k+1}$ that are left of $\la_1-b_1$ to the right of $\la_1-b_1$. Note that
	$\la_2,\dots,\la_2,\la_2-b_j,\dots,\la_2-b_{k+1}$ are left of $\la_2-b_1$ or $\la_2-b_1+1$ and we leave them in this position.
	
	Thus we conclude that it is enough to prove that
	\eq
	\label{to prove}
	\|\mu+\rho\|^2>\|\nu+\rho\|^2,
	\eeq
	where
	\begin{gather*}
		\mu=(\la_1-b_1,\underbrace{\la_2}_{q-2j},\la_2-b_j,\dots,\la_2-b_{k+1},\underbrace{\la_2-b_1}_{2k-1},\la_{q+1},\dots,\la_n);\\
		\nu=(\la_1-b_1+1,\underbrace{\la_2}_{q-2j},\la_2-b_j,\dots,\la_2-b_{k+1},\underbrace{\la_2-b_1+1}_{2k-1},\la_{q+1},\dots,\la_n).
	\end{gather*}
	
	Proving \eqref{to prove} is equivalent to proving that the expression
	\begin{multline*}
		[(\la_1-b_1+\rho_1)^2-(\la_1-b_1+1+\rho_1)^2]+[(\la_2-b_1+\rho_{q-2k+2})^2-(\la_2-b_1+1+\rho_{q-2k+2})^2]+\\
		\dots+[(\la_2-b_1+\rho_q)^2-(\la_2-b_1+1+\rho_q)^2]
	\end{multline*}
	is positive. Factoring each difference of squares, we see this is equivalent to the expression
	\[
	(2\la_1-2b_1+1+2\rho_1)+(2\la_2-2b_1+1+2\rho_{q-2k+2})+\\
	\dots+(2\la_2-2b_1+1+2\rho_q)
	\]
	being negative. Dividing by two and simplifying we see that we should prove
	\eq
	\label{to prove 2}
	\la_1+(2k-1)\la_2-2kb_1+k+\rho_1+\rho_{q-2k+2}+\dots+\rho_q<0.
	\eeq
	We compute
	\begin{multline*}
		\rho_{q-2k+2}+\dots+\rho_q=(n-q+2k-2)+\dots+(n-q)=\\
		(2k-1)(n-q)+1+2+\dots+(2k-2)=(2k-1)(n-q)+(k-1)(2k-1).
	\end{multline*}
	On the other hand, since \eqref{basic di so c1} holds strictly for $\la$ and since $\la_1>\la_2$, we see that
	\[
	\la_1+(2k-1)\la_2<k(\la_1+\la_2)<k(-2n+q+2).
	\]
	Thus we see that to prove \eqref{to prove 2} it is enough to prove that
	\[
	k(-2n+q+2) -2kb_1+k+n-1+(2k-1)(n-q)+(k-1)(2k-1)\leq 0.
	\]
	Simplifying and taking into account that $b_1\geq 1$, we see that the last inequality follows if we prove
	\[
	(q-2k)(-k+1)\leq 0,
	\]
	but this is obvious since $2k\leq 2j\leq q$ and since $k\geq 1$.
	
	It remains to prove the theorem when $2j>q$. In that case we set \[
	\la'=(\la-s_j)^+;\qquad s'=s-s_j.
	\]
	We are going to prove
	\eq
	\label{to prove 3}
	\|(\la-s)^++\rho\|^2-\|\la+\rho\|^2 > \|(\la'-s')^++\rho\|^2-\|\la'+\rho\|^2.
	\eeq
	This implies that the statement of the theorem holds for $\la$ and $s$ if it holds for $\la'$ and $s'$. 
	On the other hand, since $\la'$ starts with coordinates 
	\[
	\la_1-1,\underbrace{\la_2-1,\dots,\la_2-1}_{q-1},
	\]
	we see that $\la_1'<\la_1$, $\la_2'<\la_2$ and $q'\geq q$, so \eqref{basic di so c1} holds strictly for $\la'$. Now if $2j'\leq q'$ we already proved the theorem for $\la'$ and $s'$, and if $2j'>q'$ we note that $s'$ is shorter than $s$ and so we can assume the theorem holds for $\la'$ and $s'$ by induction.
	
	To prove \eqref{to prove 3}, we first note that Lemma \ref{gen prv} immediately implies
	\[
	\|(\la-s)^++\rho\|^2 \geq \|(\la'-s')^++\rho\|^2,
	\]
	by setting $\mu=\la'$ and $\nu=s'$, and picking $w_1$ such that $\la-s_j=w_1\mu$ and $w_2=1$. 
	So it is enough to prove that
	\eq
	\label{to prove 4}
	\|\la+\rho\|^2 < \|\la'+\rho\|^2.
	\eeq
	The difference $\|\la+\rho\|^2 - \|\la'+\rho\|^2$ is the sum of expressions
	\[
	(\la_i+\rho_i)^2-(\la_i-1+\rho_i)^2=2\la_i-1+2\rho_i,
	\]
	where $i$ runs over $1,2,\dots,q$ and over some $2j-q$ values greater than $q$. For $i>q$, we use the strict \eqref{basic di so c1} and $\la_1>\la_2$ to conclude
	\[
	2\la_i-1+2\rho_i<2\la_2-1+2\rho_q<(-2n+q+2)-1+2n-2q=-q+1<0.
	\]
	Furthermore, we claim that
	\[
	\sum_{i=1}^q (2\la_i-1+2\rho_i)=2\la_1+(2q-2)\la_2-q+2(\rho_1+\dots+\rho_q)
	\]
	is also negative; this will then imply \eqref{to prove 4}.
	Using the strict \eqref{basic di so c1} and $\la_1>\la_2$, we see that
	\[
	2\la_1+(2q-2)\la_2<q(\la_1+\la_2)<q(-2n+q+2)=-2qn+q^2+2q.
	\]
	On the other hand,
	\[
	2(\rho_1+\dots+\rho_q)=2qn-2(1+\dots+q)=2qn-q(q+1).
	\]
	So
	\[
	\sum_{i=1}^q (2\la_i-1+2\rho_i)< (-2nq+q^2+2q) -q + (2qn-q^2-q)=0.
	\]
	This finishes the proof of Theorem \ref{lem so c1}. 
	
	\smallskip
	
	We now turn to Case 2, i.e., 
	\[
	\la=(\underbrace{\la_1,\dots,\la_1}_p,\la_{p+1},\dots,\la_n)
	\]
	for some $p\in[2,n]$, with $\la_1>\la_{p+1}$ if $p<n$. Besides the basic inequality \eqref{basic di so c2}, we also examine when
	\eq
	\label{di si}
	\|(\la-s_i)^++\rho\|^2\geq \|\la+\rho\|^2
	\eeq
	for a basic Schmid module $s_i$ with $2i\leq p$. Since $(\la-s_i)^+=\la-(\eps_{p-2i+1}+\dots+\eps_p)$, \eqref{di si} is equivalent to
	\eq
	\label{di si 2}
	2\langle\la+\rho,\eps_{p-2i+1}+\dots+\eps_p\rangle\leq \|\eps_{p-2i+1}+\dots+\eps_p\|^2=2i.
	\eeq
	Furthermore,
	\[
	\langle\la+\rho,\eps_{p-2i+1}+\dots+\eps_p\rangle =2i\la_1+2i(n-p)+(1+\dots+(2i-1))=2i(\la_1+n-p+i-\frac{1}{2}).
	\]
	We substitute this into \eqref{di si 2} and divide the resulting inequality by $4i$. It follows that \eqref{di si} is equivalent to
	\eq
	\label{disc pts so}
	\la_1\leq -n+p-i+1.
	\eeq
	In particular, for $i=1$ this is the basic inequality \eqref{basic di so c2}.

	\bigskip
	
	\begin{thm}
		\label{lem unit nonunit so} Let $\la$ be in Case 2, i.e.,
		\[
		\la=(\underbrace{\la_1,\dots,\la_1}_p,\la_{p+1},\dots,\la_n)
		\]
		for some $p\in[2,n]$, with $\la_1>\la_{p+1}$ if $p<n$. Then:
		\begin{enumerate}
			\item
			If for some integer $i\in[1,[\frac{p}{2}]]$ 
			\[
			\lambda_1< -n+p-i+1,
			\]
			then the inequality 
			\eq
			\label{din1 so}
			\|(\la-s)^++\rho\|^2 \geq \|\la+\rho\|^2
			\eeq
			holds strictly for any Schmid module $s$ with at most $2i$ nonzero components.
			\item
			If 
			\eq
			\label{cont so}
			\lambda_1< -n+p-\left[\frac{p}{2}\right]+1=-n+\left[\frac{p+1}{2}\right]+1,
			\eeq
			then \eqref{din1 so} holds strictly for any Schmid module $s$.
		\end{enumerate}
	\end{thm}
	\pf
	(1) Let $i\in[1,[\frac{p}{2}]]$ be an integer such that  
	\eq
	\label{ind step so 1}
	\lambda_1< -n+p-i+1,
	\eeq
	and let 
	\[
	s=(b_1,b_1,\dots,b_j,b_j,0,\dots,0)
	\]
	with $j\leq i$ and $b_1\geq\dots\geq b_j>0$.
	
	We need to show that \eqref{din1 so} holds strictly for $\la$ and $s$.
	
	Let $k\in[1,j]$ be such that
	\[
	b_1=\dots=b_k>b_{k+1}.
	\]
	Let $s'=s-s_k$. It is enough to prove
	\eq
	\label{ind step so 2}
	\|(\la-s)^++\rho\|^2>\|(\la-s')^++\rho\|^2;
	\eeq
	the statement then follows by induction on $b_1$. (If $b_1=1$, then $s'=0$, and \eqref{ind step so 2} is the same as the strict \eqref{din1 so}.)
	
	We first note that $(\la-s)^+$ contains coordinates
	\[
	\underbrace{\la_1,\dots,\la_1}_{p-2j},\underbrace{\la_1-b_i,\dots,\la_1-b_{k+1}}_{2j-2k},\underbrace{\la_1-b_1,\dots,\la_1-b_1}_{2k}
	\]
	in that order, and then also $\la_{p+1},\dots,\la_n$, which may be interlaced with these coordinates.
	Similarly, $(\la-s')^+$ contains coordinates
	\[
	\underbrace{\la_1,\dots,\la_1}_{p-2j},\underbrace{\la_1-b_i,\dots,\la_1-b_{k+1}}_{2j-2k},\underbrace{\la_1-b_1+1,\dots,\la_1-b_1+1}_{2k}
	\]
	in that order, and then also $\la_{p+1},\dots,\la_n$, which may be interlaced with these coordinates.
	
	Using Lemma \ref{red sp} (1), we may assume that $\la_{p+1},\dots\la_n$ are all the way to the right in both $(\la-s)^+$ and $(\la-s')^+$. Thus it suffices to show that
	\[
	\sum_{r=p-2k+1}^p \left[(\la_1-b_1+\rho_r)^2-(\la_1-b_1+1+\rho_r)^2\right] >0.
	\]
	By factoring differences of squares, this is equivalent to
	\eq
	\label{ind step so 3}
	\sum_{r=p-2k+1}^p (2\la_1-2b_1+1+2\rho_r) <0.
	\eeq
	Since
	\begin{multline*}
		\sum_{r=p-2k+1}^p\rho_r=(n-p+2k-1)+\dots+(n-p)=\\ 
		2k(n-p-1)+(2k+\dots+1)=2k(n-p-1)+k(2k+1),
	\end{multline*}
	\eqref{ind step so 3} is equivalent to
	\[
	2k(2\la_1-2b_1+1)+4k(n-p-1)+2k(2k+1)<0,
	\]
	which upon dividing by $4k$ becomes
	\[
	\la_1-b_1+n-p+k <0.
	\]
	Since $b_1\geq 1$ and since $k\leq j\leq i$, this follows from our assumption \eqref{ind step so 1}. This finishes the proof of (1).
	
	\smallskip
	
	(2) If $s$ has at most $p$ nonzero components, then the statement follows from (1) by specializing to $i=[\frac{p}{2}]$. So we can assume that 
	\[
	s=(b_1,b_1,\dots,b_i,b_i,0,\dots,0),
	\]
	with $2i>p$.
	
	Let $\la'=(\la-s_i)^+$. Then $\la'_1,\dots,\la'_p$ are all equal to $\la_1-1$, while for $r>p$, $\la'_r$ is equal to either $\la_r$ or $\la_r-1$. In particular, $\la'$ is still in Case 2, with $p'\geq p$ and with $\la_1'=\la_1-1$, and so we have
	\[
	\la_1' <-n+[\frac{p'+1}{2}]+1.
	\]
	We claim that
	\eq
	\label{ind step so 4}
	\|(\la-s)^++\rho\|^2-\|\la+\rho\|^2 > \|(\la'-s')^++\rho\|^2-\|\la'+\rho\|^2.
	\eeq
	If we prove \eqref{ind step so 4}, then we can use induction. Namely $s'$  either has at most $p$ and hence at most $p'$ nonzero coordinates, so the statement is true by (1), or we can use the fact that the first coordinate $b_1'$ of $s'$ is strictly smaller than the first coordinate $b_1$ of $s$ and do induction on $b_1$. (As before, if $b_1=1$, then $s'=0$ and \eqref{ind step so 4} is the same as the strict \eqref{din1 so}.)
	
	To prove \eqref{ind step so 4}, we first note that
	\[
	\|(\la-s)^++\rho\|^2\geq\|(\la'-s')^+ + \rho\|^2.
	\] 
	This follows from Lemma \ref{gen prv}, if we set $\mu=\la'$, $\nu=s'$ and $w_2=1$, and pick $w_1$ such that $\la-s_i=w_1\la'$. It thus suffices to prove that
	\eq
	\label{ind step so 5}
	\|\la+\rho\|^2 < \|\la'+\rho\|^2.
	\eeq
	To prove \eqref{ind step so 5}, we note that $\|\la+\rho\|^2 - \|\la'+\rho\|^2$ is the sum of expressions of the form
	\[
	(\la_r+\rho_r)^2-(\la_r-1+\rho_r)^2=2\la_r-1+2\rho_r,
	\]
	with summation over $r=1,\dots,p$ and over some $r>p$. Note that since $2i>p$, there is at least one summand with $r>p$.
	
	Since $\la_1,\dots,\la_r$ are all equal to $\la_1$, and since
	\[
	\sum_{r=1}^p\rho_r=(n-1)+\dots+(n-p)=pn-\frac{p(p+1)}{2},
	\]
	we see that
	\[
	\sum_{r=1}^p(2\la_r-1+2\rho_r)=2p\la_1-p+2pn-p(p+1)=p(2\la_1-2+2n-p).
	\]
	Since $\la$ satisfies \eqref{cont so}, it follows that 
	\begin{multline}
		\label{ind step so 6}
		\sum_{r=1}^p(2\la_r-1+2\rho_r) < p\left[(-2n+2p-2\left[\frac{p}{2}\right]+2)-2+2n-p\right]= \\
		p\left(p-2\left[\frac{p}{2}\right]\right)\leq p.
	\end{multline}
	On the other hand, for any $r>p$ we have 
	\[
	\la_r\leq\la_1-1 < -n+p-\left[\frac{p}{2}\right],
	\]
	and $\rho_r\leq\rho_{p+1}=n-p-1$. It follows that
	\eq
	\label{ind step so 7}
	2\la_r-1+2\rho_r < -2n+2p-2\left[\frac{p}{2}\right]-1+2(n-p-1)=-2\left[\frac{p}{2}\right]-3 < -p.
	\eeq
	Since $\|\la+\rho\|^2 - \|\la'+\rho\|^2$ is the sum of expressions $2\la_r-1+2\rho_r$ over $r=1,\dots,p$ and over some (at least one) $r>p$, we see from \eqref{ind step so 6} and \eqref{ind step so 7} that
	\[
	\|\la+\rho\|^2 - \|\la'+\rho\|^2 < p-p=0,
	\]
	as claimed. 
	\epf

	\subsection{Dirac inequality for $\frsu(p,q)$, $p\leq q$}
	The basic Schmid $\frk$-submodules of $S(\frp^-)$ have lowest weights $-s_i$, where
	\[
	s_i=(\underbrace{1,\dots,1}_i,0,\dots,0\bbar 0,\dots,0,\underbrace{-1,\dots,-1}_i)
	\] 
	for $i=1,\dots,p$.
	Moreover, all irreducible $\frk$-submodules of $S(\frp^-)$ have lowest weights $-s$, where
	\eq
	\label{gen schmid su}
	s=(b_1,\dots,b_p\,|\,0,\dots,0,-b_p,\dots,-b_1),
	\eeq
	where $b_1\geq\dots\geq b_p\geq 0$ are integers.

	The highest weight $(\frg,K)$-modules have highest weights of the form
	\[
	\lambda=(\lambda_1,\dots,\lambda_p\,|\,\lambda_{p+1},\dots,\lambda_n),
	\]
	where components $\la_1,\dots,\la_n$ satisfy
	\[
	\la_1\geq\dots\geq\la_p;\qquad \la_{p+1}\geq\dots\geq\la_n,
	\]
	and $\la_i-\la_j$ is an integer for any $i,j\in\{1,\dots,p\}$ or $i, j \in \{p + 1, \dots, n\}$.

	In this case $\rho = \frac{1}{2}(n-1, n-3, \ldots, -n + 3, -n + 1).$
	
	The basic necessary condition for unitarity is, as before, the Dirac inequality
	\eq
	\label{basic di su}
	\|(\la-s_1)^++\rho\|^2\geq\|\la+\rho\|^2.
	\eeq
	To understand this inequality better, let $p'\leq p$ and $q'\leq q$ be the maximal positive integers such that
	\[
	\lambda_1=\dots=\lambda_{p'}\qquad\text{and}\qquad  \lambda_{n-q'+1}=\dots=\lambda_n. 
	\]
	Then 
	\[
	(\la-s_1)^+=\la-(\eps_{p'}-\eps_{n-q'+1}).
	\]

	The inequality \eqref{basic di su} now becomes equivalent to
	$\|\la+\rho-\gamma\|^2\geq\|\la+\rho\|^2$, or to
	$2\langle \la+\rho,\gamma\rangle \leq \|\gamma\|^2=2$, or to
	$\langle\la+\rho,\eps_{p'}-\eps_{n-q'+1}\rangle\leq 1$. Since $\la_{p'}=\la_1$, $\la_{n-q'+1}=\la_n$, and since $\rho_{p'}-\rho_{n-q'+1}=n-q'+1-p'$, we see that \eqref{basic di su} is equivalent to
	\eq
	\label{basic di su 2}
	\la_1-\la_n\leq -n+p'+q'.
	\eeq

	As in the other cases, 
	we start by examining the condition on $\la$ which ensures
	that the Dirac inequality
	\eq
	\label{din1 su}
	\|(\la-s_i)^++\rho\|^2\geq\|\la+\rho\|^2
	\eeq
	holds, where $i=1,\dots,\min(p',q')$. We already know that for $i=1$ this is just the basic inequality \eqref{basic di su}, or equivalently \eqref{basic di su 2}.
	
	To examine when \eqref{din1 su} holds, we first note that
	\[
	(\la-s_i)^+ = \la-t_i,\quad \text{where}\quad t_i=(\eps_{p'-i+1}+\dots+\eps_{p'})-(\eps_{n-q'+1}+\dots+\eps_{n-q'+i}).
	\]
	Since $\|t_i\|^2=2i$, it follows that \eqref{din1 su} is equivalent to
	\eq
	\label{din2 su}
	\langle\la+\rho,t_i\rangle\leq i.
	\eeq
	We note that $\la_{p'-i+1},\dots,\la_{p'}$ are all equal to $\la_1$ while $\la_{n-q'+1},\dots,\la_{n-q'+i}$ are all equal to $\la_n$, and that $\rho_{p'-i+j}-\rho_{n-q'+j}=n-p'-q'+i$ for any $j \in [1,i]$. Plugging this into \eqref{din2 su} and dividing by $i$, we see that \eqref{din2 su} (and hence \eqref{din1 su}) is equivalent to
	\eq
	\label{din3 su}
	\la_1-\la_n\leq -n+p'+q'-i+1.
	\eeq

	\begin{thm}
		\label{lem unit nonunit su} Let 
		\[
		\la=(\underbrace{\la_1,\dots,\la_1}_{p'},\la_{p'+1},\dots,\la_p\bbar\la_{p+1},\dots,\la_{n-q'},\underbrace{\la_n,\dots,\la_n}_{q'}).
		\] 
		Then:
		\begin{enumerate}
			\item
			If $\la$ satisfies \eqref{din3 su} strictly for some integer $i\in[1,\min(p',q')]$, then the inequality
			\eq 
			\label{din4 su}
			\|(\la-s)^++\rho\|^2\geq\|\la+\rho^2\|
			\eeq
			holds strictly for any Schmid module 
			\eq
			\label{din4a su}
			s=(b_1,\dots,b_j,0,\dots,0\bbar 0,\dots,0,-b_j,\dots,-b_1)
			\eeq
			with $j\leq i$.
			\item
			If 
			\eq
			\label{cont su}
			\lambda_1-\la_n< -n+p'+q'-\min(p',q')+1=-n+\max(p',q')+1,
			\eeq
			then \eqref{din4 su} holds strictly for any Schmid module $s$.
		\end{enumerate}
	\end{thm}
	\pf
	Both (1) and (2) will follow from the following discussion.  Let $s$ be any Schmid module:
	\[
	s=(b_1,\dots,b_p\bbar -b_q,\dots,-b_{p+1},-b_p,\dots,-b_1)
	\]
	where $b_1\geq\dots\geq b_p\geq 0$ are integers, not all of them zero, and $b_{p+1}=\dots=b_q=0$. Let $k\in[1,p]$ be the maximal integer such that $b_1=\dots=b_k$, and define
	\[
	s'=s-s_k=(\underbrace{b_1-1}_k,b_{k+1},\dots,b_p\bbar -b_q,\dots,-b_{k+1},\underbrace{-b_1+1}_k).
	\]
	For both (1) and (2), it will suffice to prove
	\eq
	\label{ind step su 2}
	\|(\la-s)^++\rho\|^2>\|(\la-s')^++\rho\|^2.
	\eeq
	The statements (1) and (2) then follow by induction on $b_1$; if $b_1=1$, then $s'=0$, and \eqref{ind step su 2} is the same as the strict \eqref{din4 su}.
	
	To prove \eqref{ind step su 2}, we examine separately the two sides of all expressions involved, the left respectively right side (of the bar).
	
	If $k\leq p'$, the left side of $(\la-s)^+$ contains coordinates
	\eq
	\label{ind step su 2a} 
	\underbrace{\la_1-b_{p'},\dots,\la_1-b_{k+1}}_{p'-k},\underbrace{\la_1-b_1,\dots,\la_1-b_1}_{k}
	\eeq
	in that order, and also $\la_{p'+1}-b_{p'+1},\dots,\la_p-b_p$, arranged in descending order and appropriately interlaced with the coordinates \eqref{ind step su 2a}. 
	
	On the other hand, the left side of $(\la-s')^+$ contains the same coordinates, except that it contains $k$ coordinates equal to $\la_1-b_1+1$ in place of $k$ coordinates equal to $\la_1-b_1$.
	
	Using a version of Lemma \ref{red sp}, we may assume that $\la_{p'+1}-b_{p'+1},\dots,\la_p-b_p$ are all the way to the right in the left sides of both $(\la-s)^+$ and $(\la-s')^+$.
	\medskip
	
	Thus the contribution of the left sides to 
	$\|(\la-s)^++\rho\|^2 - \|(\la-s')^++\rho\|^2$
	is
	\begin{multline}
		\label{ind step su 2b}
		\sum_{u=p'-k+1}^{p'} \left[(\la_1-b_1+\rho_u)^2-(\la_1-b_1+1+\rho_u)^2\right] = \\
		-\sum_{u=p'-k+1}^{p'} (2\la_1-2b_1+1+2\rho_u) = 
		-\left(2k\la_1-2kb_1+k+\sum_{u=p'-k+1}^{p'}2\rho_u\right)\geq\\
		-2k\la_1+k-\sum_{u=p'-k+1}^{p'}2\rho_u
	\end{multline}
	(for the last inequality we used $b_1\geq 1$).
	
	\medskip
	
	If $k> p'$, the left side of $(\la-s)^+$ contains coordinates
	\eq
	\label{ind step su 2c} 
	\underbrace{\la_1-b_1,\dots,\la_1-b_1}_{p'},\la_{p'+1}-b_1,\dots,\la_k-b_1
	\eeq
	in that order, and also $\la_{k+1}-b_{k+1},\dots,\la_p-b_p$, arranged in descending order and appropriately interlaced with the coordinates \eqref{ind step su 2c}. 
	
	On the other hand, the left side of $(\la-s')^+$ contains coordinates
	\eq
	\label{ind step su 2d} 
	\underbrace{\la_1-b_1+1,\dots,\la_1-b_1+1}_{p'},\la_{p'+1}-b_1+1,\dots,\la_k-b_1+1
	\eeq
	in that order, and also $\la_{k+1}-b_{k+1},\dots,\la_p-b_p$, arranged in descending order and appropriately interlaced with the coordinates \eqref{ind step su 2d}. These last coordinates are the same as in $(\la-s)^+$, and also at the same places.
	
	Using (the extension of) Lemma \ref{red sp}, we may assume that $\la_{k+1}-b_{k+1},\dots\la_p-b_p$ are all the way to the right in the left sides of both $(\la-s)^+$ and $(\la-s')^+$.
	Thus the contribution of the left sides to 
	$\|(\la-s)^++\rho\|^2 - \|(\la-s')^++\rho\|^2$
	is
	\begin{multline}
		\label{ind step su 2e}
		\sum_{u=1}^{p'} \left[(\la_1-b_1+\rho_u)^2-(\la_1-b_1+1+\rho_u)^2\right] +\\
		\sum_{u=p'+1}^k \left[(\la_u-b_1+\rho_u)^2-(\la_u-b_1+1+\rho_u)^2\right] = \\
		-\sum_{u=1}^{p'} (2\la_1-2b_1+1+2\rho_u)-\sum_{u=p'+1}^{k} (2\la_u-2b_1+1+2\rho_u) 
	\end{multline}
	
	Since $\la_u\leq\la_1-1$ for $u>p'$, and since $b_1\geq 1$, we conclude that the expression \eqref{ind step su 2e} is 
	\eq
	\label{ind step su 2f}
	\geq -2k\la_1+3k-2p'-\sum_{u=1}^{k}2\rho_u.
	\eeq
	
	\medskip
	
	We now consider the right sides. We first assume $k\leq q'$. Then the right side of $(\la-s)^+$ contains coordinates
	\eq
	\label{ind step su 2g}
	\underbrace{\la_n+b_1,\dots,\la_n+b_1}_k,\la_n+b_{k+1},\dots,\la_n+b_{q'},
	\eeq
	in that order, and also $\la_{p+1}+b_q,\dots,\la_{n-q'}+b_{q'+1}$, arranged in descending order and appropriately interlaced with coordinates \eqref{ind step su 2g}.
	On the other hand, the right side of $(\la-s')^+$ contains the same coordinates, except that it contains $k$ coordinates equal to $\la_n+b_1-1$ in place of $k$ coordinates equal to $\la_n+b_1$.
	
	Using (the extension of) Lemma \ref{red sp}, we may assume that $\la_{p+1}+b_q,\dots,\la_{n-q'}+b_{q'+1}$ are all the way to the left in the right groups of both $(\la-s)^+$ and $(\la-s')^+$. Thus the contribution of the right sides to 
	$\|(\la-s)^++\rho\|^2 - \|(\la-s')^++\rho\|^2$ is
	\begin{multline}
		\label{ind step su 4}
		\sum_{v=n-q'+1}^{n-q'+k} \left[(\la_n+b_1+\rho_v)^2-(\la_n+b_1-1+\rho_v)^2\right] = \\
		\sum_{v=n-q'+1}^{n-q'+k} (2\la_n+2b_1-1+2\rho_v)=  
		2k\la_n+2kb_1-k+\sum_{v=n-q'+1}^{n-q'+k} 2\rho_v\geq \\
		2k\la_n+k+\sum_{v=n-q'+1}^{n-q'+k} 2\rho_v
	\end{multline}
	(for the last inequality we used $b_1\geq 1$).
	
	\medskip
	
	If $k> q'$, the right side of $(\la-s)^+$ contains coordinates
	\eq
	\label{ind step su 5}
	\la_{n-k+1}+b_1,\dots,\la_{n-q'}+b_1,\underbrace{\la_n+b_1,\dots,\la_n+b_1}_{q'},
	\eeq
	in that order, and also $\la_{p+1}+b_q,\dots,\la_{n-k}+b_{k+1}$, arranged in descending order and appropriately interlaced with coordinates \eqref{ind step su 5}.
	
	On the other hand, the right side of $(\la-s')^+$ contains coordinates
	\eq
	\label{ind step su 6} 
	\la_{n-k+1}+b_1-1,\dots,\la_{n-q'}+b_1-1,\underbrace{\la_n+b_1-1,\dots,\la_n+b_1-1}_{q'},
	\eeq
	in that order, and also $\la_{p+1}+b_q,\dots,\la_{n-k}+b_{k+1}$, arranged in descending order and appropriately interlaced with the coordinates \eqref{ind step su 6}. These last coordinates are the same as in $(\la-s)^+$, and also at the same places.
	
	Using (the extension of) Lemma \ref{red sp}, we may assume that $\la_{p+1}+b_q,\dots,\la_{n-k}+b_{k+1}$ are all the way to the left in the right groups of both $(\la-s)^+$ and $(\la-s')^+$. Thus the contribution of the right sides to 
	$\|(\la-s)^++\rho\|^2 - \|(\la-s')^++\rho\|^2$ is
	\begin{multline}
		\label{ind step su 7}
		\sum_{v=n-k+1}^{n-q'} \left[(\la_v+b_1+\rho_v)^2-(\la_v+b_1-1+\rho_v)^2\right] + \\
		\sum_{v=n-q'+1}^{n} \left[(\la_n+b_1+\rho_v)^2-(\la_n+b_1-1+\rho_v)^2\right]=\\
		\sum_{v=n-k+1}^{n-q'} (2\la_v+2b_1-1+2\rho_v) + 
		\sum_{v=n-q'+1}^{n} (2\la_n+2b_1-1+2\rho_v).
	\end{multline}
	
	Since $\la_v\geq\la_n+1$ for $v\leq n-q'$, and since $b_1\geq 1$, we conclude that the expression \eqref{ind step su 7} is 
	\eq
	\label{ind step su 8}
	\geq 2k\la_n+3k-2q'+\sum_{v=n-k+1}^{n}2\rho_v.
	\eeq
	
	\medskip
	
	Let us now assume that $k\leq\min(p',q')$; this is always true under the assumptions of (1). Using \eqref{ind step su 2b} and \eqref{ind step su 4}, we see that to prove the required inequality \eqref{ind step su 2} it is enough to prove that
	\eq
	\label{ind step su 9}
	2k(\la_1-\la_n)-2k+\sum_{u=p'-k+1}^{p'}2\rho_u-\sum_{v=n-q'+1}^{n-q'+k}2\rho_v < 0.
	\eeq
	Since for any integer $t\in[1,k]$ we have
	\[
	\rho_{p'-k+t}-\rho_{n-q'+t}=n-p'-q'+k,
	\]
	we see that 
	\[
	\sum_{u=p'-k+1}^{p'}2\rho_u-\sum_{v=n-q'+1}^{n-q'+k}2\rho_v  =  2k(n-p'-q'+k).
	\]
	We substitute this into \eqref{ind step su 9} and divide by $2k$. It follows that
	\eqref{ind step su 9} is equivalent to 
	\eq
	\label{ind step su 10}
	\la_1-\la_n-1+n-p'-q'+k <0.
	\eeq
	Under the assumptions of (1), it follows that the left side of \eqref{ind step su 10} is
	\[
	<(-n+p'+q'-i+1)-1+n-p'-q'+k = -i+k\leq 0,
	\]
	so this finishes the proof of (1).
	
	\medskip
	
	Under the assumptions of (2), and our current assumption that $k\leq\min(p',q')$, it follows that the left side of \eqref{ind step su 10} is
	\[
	<(-n+\max(p',q')+1)-1+n-p'-q'+k \leq \max(p',q')-p'-q'+\min(p',q') = 0,
	\]
	so this finishes the proof of (2) in case $k\leq\min(p',q')$.
	
	\medskip
	
	Let us now assume that $p'<k\leq q'$. 
	Using \eqref{ind step su 2f} and \eqref{ind step su 4}, we see that to prove the required inequality \eqref{ind step su 2} it is enough to prove that
	\eq
	\label{ind step su 11}
	2k(\la_1-\la_n)-4k+2p'+\sum_{u=1}^{k}2\rho_u-\sum_{v=n-q'+1}^{n-q'+k}2\rho_v < 0.
	\eeq
	Since for any integer $t\in[1,k]$ we have
	\[
	\rho_{t}-\rho_{n-q'+t}=n-q',
	\]
	we see that 
	\[
	\sum_{u=1}^{k}2\rho_u-\sum_{v=n-q'+1}^{n-q'+k}2\rho_v  =  2k(n-q').
	\]
	We substitute this into \eqref{ind step su 11}. It follows that
	\eqref{ind step su 11} is equivalent to 
	\eq
	\label{ind step su 12}
	2k(\la_1-\la_n)-4k+2p'+2k(n-q') <0.
	\eeq
	
	Under the assumptions of (2), remembering that in the present case $\max(p',q')=q'$, we see that the left side of \eqref{ind step su 10} is
	\[
	<2k(-n+q'+1)-4k+2p'+2k(n-q') = -2k+2p',
	\]
	and this is $< 0$ since in the present case $k>p'$. This finishes the proof of (2) in case $p'<k\leq q'$.
	
	\medskip
	
	Let us now assume that $q'<k\leq p'$. 
	Using \eqref{ind step su 2b} and \eqref{ind step su 8}, we see that to prove the required inequality \eqref{ind step su 2} it is enough to prove that
	\eq
	\label{ind step su 13}
	2k(\la_1-\la_n)-4k+2q'+\sum_{u=p'-k+1}^{p'}2\rho_u-\sum_{v=n-k+1}^{n}2\rho_v < 0.
	\eeq
	Since for any integer $t\in[1,k]$ we have
	\[
	\rho_{p'-k+t}-\rho_{n-k+t}=n-p',
	\]
	we see that 
	\[
	\sum_{u=p'-k+1}^{p'}2\rho_u-\sum_{v=n-k+1}^{n}2\rho_v  =  2k(n-p').
	\]
	We substitute this into \eqref{ind step su 13}. It follows that
	\eqref{ind step su 13} is equivalent to 
	\eq
	\label{ind step su 14}
	2k(\la_1-\la_n)-4k+2q'+2k(n-p') <0.
	\eeq
	
	Under the assumptions of (2), remembering that in the present case $\max(p',q')=p'$, we see that the left side of \eqref{ind step su 14} is
	\[
	<2k(-n+p'+1)-4k+2q'+2k(n-p') =-2k+2q',
	\]
	and this is $< 0$ since in the present case $k>q'$. 
	This finishes the proof of (2) in case $q'<k\leq p'$.
	
	\medskip
	
	Finally, let us assume that $k>\max(p',q')$. 
	Using \eqref{ind step su 2f} and \eqref{ind step su 8}, we see that to prove the required inequality \eqref{ind step su 2} it is enough to prove that
	\eq
	\label{ind step su 15}
	2k(\la_1-\la_n)-6k+2p'+2q'+\sum_{u=1}^{k}2\rho_u-\sum_{v=n-k+1}^{n}2\rho_v < 0.
	\eeq
	
	Since for any integer $t\in[1,k]$ we have
	\[
	\rho_{t}-\rho_{n-k+t}=n-k,
	\]
	we see that 
	\[
	\sum_{u=1}^{k}2\rho_u-\sum_{v=n-k+1}^{n}2\rho_v  =  2k(n-k).
	\]
	We substitute this into \eqref{ind step su 15}. It follows that
	\eqref{ind step su 15} is equivalent to 
	\eq
	\label{ind step su 16}
	2k(\la_1-\la_n)-6k+2p'+2q'+2k(n-k) <0.
	\eeq
	Under the assumptions of (2), it follows that the left side of \eqref{ind step su 16} is
	\begin{multline*}
		<2k(-n+\max(p',q')+1)-6k+2p'+2q'+2k(n-k)=\\
		2k(\max(p',q')-k)-4k+2p'+2q',
	\end{multline*}
	and this is $<0$, since in the present case $k>\max(p',q')$. 
	This finishes the proof of (2) in case $k>\max(p',q')$, and hence the proof of Theorem \ref{lem unit nonunit su} is completed.
	\epf

	\subsection{Dirac inequality for $\frso(2, 2n-2)$}
	The basic Schmid $\frk$-submodules of $S(\frp^-)$ have lowest weights $-s_1$ or $-s_2$, where
	\[
	s_1=(1, 1, 0, \ldots 0), \
	s_2 = (2, 0, 0, \ldots 0).
	\] 
	
	Moreover, all irreducible $\frk$-submodules of $S(\frp^-)$ have lowest weights $-s_{a,b}$, where
	$s_{a,b} = (2b + a, a, 0, \ldots 0).$

	The highest weight $(\frg,K)$-modules have highest weights of the form
	$\lambda = (\lambda_1, \ldots, \lambda_n)$ is $\frk$-dominant if 
	\[
	\lambda_2 \geq \lambda_3 \geq \cdots \lambda_{n-1} \geq |\lambda_n|.
	\]
	where components $\la_1,\dots,\la_n$ satisfy
	$\lambda_i - \lambda_j \in \bbZ$ for $2\leq i, j \leq n.$

	In this case $\rho = (n-1, n-2, \ldots, 0)$.
	
	The basic necessary condition for unitarity is, as before, the Dirac inequality
	\eq
	\label{basic di so2, 2n-2}
	\|(\la-s_1)^++\rho\|^2\geq\|\la+\rho\|^2.
	\eeq

	The basic Dirac inequality for a Schmid module $s$ is
	\begin{equation}\label{eq:so_even_basic}
		\| (\lambda - s)^+ + \rho \|^2 \geq \| \lambda + \rho \|^2
	\end{equation}
	This is equivalent to
	\begin{equation}
		2 \langle \gamma \,|\, \lambda + \rho \rangle \leq \|\gamma\|^2
	\end{equation}
	where $\gamma$ is defined by $(\lambda - s_{a,b})^+ = \lambda - \gamma.$
	
	\begin{lem}
		The basic Dirac inequality for $s = s_1$ is given by
		\begin{equation}\label{eq:so_even_basic_dirac}
			\begin{aligned}
				\lambda_1 &\leq 0 & \text{ for } \lambda=(\lambda_1, 0,\ldots, 0)  \\
				\lambda_1 &\leq 3/2-n & \text{ for } \lambda=(\lambda_1, 1/2,\ldots, \pm 1/2) \\
				\lambda_1 + \lambda_2 &\leq 2 + p - 2n &\text{ for } \lambda=(\lambda_1, \lambda_2, \ldots, \lambda_p, \ldots, \lambda_n) \\
				&  & \text{ where } 1 \leq \lambda_2 = \cdots = \lambda_p > \lambda_{p+1} \text{ and } 2 \leq p \leq n.
			\end{aligned}
		\end{equation}
	\end{lem}
	\begin{proof}
		\textbf{Case 1:} $\lambda = (\lambda_1, 0, \ldots, 0)$
		
		In this case we have
		\begin{align*}
			(\lambda - s_1)^+ &= (\lambda_1 - 1, -1, 0, \ldots, 0)^+ \\
			&= (\lambda_1 - 1, 1, 0, \ldots, 0) \\
			&= \lambda - (\epsilon_1 - \epsilon_2)
		\end{align*}
		which shows that $\gamma = \epsilon_1 - \epsilon_2$ and \eqref{eq:so_even_basic} reduces to 
		\(
		\lambda_1 + n - 1 -  (n - 2) \leq 1
		\)
		which is equivalent to $\lambda_1 \leq 0.$
		
		\medskip
		
		\textbf{Case 2:} $\lambda = (\lambda_1, 1/2, \ldots, \pm 1/2)$
		
		In this case we have 
		\begin{align*}
			(\lambda - s_1)^+ &= (\lambda_1 - 1, -1/2, 1/2, \ldots,1/2, \pm 1/2)^+ \\
			&= (\lambda_1 - 1, 1/2, 1/2, \ldots,1/2, \mp 1/2) \\
			&= \lambda - (\epsilon_1 \pm \epsilon_n)
		\end{align*}
		Plugging $\gamma = \epsilon_1 \pm \epsilon_n$ into \eqref{eq:so_even_basic} we obtain $\lambda_1 + n - 1 + 1/2 \leq 1$ which gives 
		\begin{equation}
			\lambda_1 \leq 3/2 - n.
		\end{equation}
		
		\textbf{Case  3:} $1 \leq \lambda_2 = \cdots = \lambda_p > \lambda_{p+1}$
		
		The basic Dirac inequality \eqref{eq:so_even_basic} for $s=s_1$ has $\gamma = \epsilon_1 + \epsilon_p$ and
		\begin{align*}
			2(\lambda_1 + n-1 + \lambda_p + n - p) &\leq 2 \\
			\lambda_1 + \lambda_2 &\leq 2 + p - 2n
		\end{align*}
		
	\end{proof}
	
	We will refer to the first case as to the \emph{scalar} case, the second case is the \emph{spinor} case and the remaining one is the \emph{general} case. In the scalar and spinor case we can actually prove Dirac inequalities directly.
	
	\begin{thm}[Scalar case] \label{scalar}
		Let $\lambda = (\lambda_1, 0, \ldots, 0)$ such that $\lambda_1 < 2-n$. Then the Dirac inequality \eqref{eq:so_even_basic} holds for any Schmid module $s$.
	\end{thm}
	\begin{proof}
		The Dirac inequality for the second basic Schmid $s_2 = 2\epsilon_1$ yields $\lambda_1 \leq 2 -n.$  
		
		Now we have
		\begin{align*}
			(\lambda - s_{a,b})^+ &= (\lambda_1 - 2b - a, -a, 0, \ldots, 0)^+ \\
			&= (\lambda_1 - 2b - a, a, 0, \ldots, 0) \\
			&= \lambda - [(2b+a)\epsilon_1 - a\epsilon_2]
		\end{align*}
		which we will plug into the Dirac inequality \eqref{eq:so_even_basic}
		\begin{align*}
			2[ (2b+a)(\lambda_1 + n -1) - a(n-2) ] &\leq (2b+a)^2 + a^2 \\
			(2b+a)(\lambda_1 + n -1) - a(n-2)  &\leq 2b^2 + 2ab + a^2 
		\end{align*}
		Using $\lambda_1 \leq 2 -n$  we see that it is sufficient to prove that 
		\begin{align*}
			(2b+a)(2 - n + n -1) - a(n-2)  &\leq 2b^2 + 2ab + a^2 \\
			2b + (3-n)a &\leq 2b^2 + 2ab + a^2 \\
			0 &\leq 2b(b-1) + 2ab + a(a + n -3).
		\end{align*}
		Since $a, b \geq 0$ and $n \geq 4$ we are done.
	\end{proof}
	
	\begin{thm}[Spinor case]
		Let $\lambda = (\lambda_1, 1/2,\ldots, \pm 1/2)$ such that $\lambda_1 \leq 3/2 - n.$ Then the Dirac inequality \eqref{eq:so_even_basic} holds for any Schmid module $s$.
	\end{thm}
	\begin{proof}
		We have
		\begin{align*}
			(\lambda - s_{a,b})^+ &= (\lambda_1 - 2b - a, 1/2 -a, 1/2, \ldots, \pm 1/2)^+ \\
			&= (\lambda_1 - 2b - a, a-1/2, 1/2, \ldots, \mp 1/2) \\
			&= \lambda - [(2b+a)\epsilon_1 -(a-1)\epsilon_2 \pm \epsilon_n]
		\end{align*}
		for $a \geq 1.$

		Dirac inequality for the case $a=0$ reads
		\begin{align*}
			2 [2b(\lambda_1 + n -1)] &\leq 4b^2 \\
			\lambda_1 + n -1 &\leq b \\
			\lambda_1 &\leq (b+1) - n
		\end{align*}
		which is satisfied (strictly) whenever $b\geq 1$ due to \eqref{eq:so_even_basic_dirac}. In the general case we get 
		\[
		2[ (2b+a)(\lambda_1 + n -1) -(a-1)(1/2 + n-2) + 1/2] \leq (2b+a)^2 + (a-1)^2 + 1 
		\]
		and using \eqref{eq:so_even_basic_dirac} it is sufficient to show
		\begin{align*}
			2[ (2b+a)(3/2-n + n -1) -(a-1)(1/2 + n-2) + 1/2] &\leq (2b+a)^2 + (a-1)^2 + 1.
		\end{align*}
		This is in turn equivalent to showing nonnegativity of 
		\begin{align*}
			a^2 + 2 a b +  a n - 3 a + 2b^2 - b - n + 2 &\geq 0 \\
			(a + b)^2 + b(b-1) + a(n-3) + 2 - n &\geq 0
		\end{align*}
		This is clearly decreasing in $b$ and so we just need to prove that 
		\[
		a^2 + (n-3)a - n + 2 \geq 0.
		\]
		The roots of this quadratic polynomial are $1$ and $2-n$ which finishes this case. 
	\end{proof}
	
	\begin{lem}
		Let $s_{a,b} = (2b + a, a, 0,\ldots, 0)$ be a Schmid module with $a \geq 1$ and $b \geq 0.$ Let $\lambda$ satisfy \eqref{eq:so_even_basic_dirac} in the case $\lambda_2 \neq 0$ and $\lambda_1 \leq 2-n$ in the scalar case. Then we have
		\begin{equation}\label{eq:so_even_induction_b}
			\| (\lambda - s_{a,b+1})^+ + \rho \|^2 - \| \lambda + \rho \|^2 \geq \| (\lambda - s_{a, b})^+ + \rho\|^2 - \| \lambda + \rho \|^2.
		\end{equation}
		%and the inequality is strict if the inequality from the assumptin is strict. 
	\end{lem}
	\begin{proof}
		Since the weights differ only in the first coordinate, the difference of the left hand side and the right hand side is just the difference of squares on the first coordinate:
		\begin{align*}
			(\lambda_1 - 2b - 2 - a + n - 1)^2 - (\lambda_1 - 2b - a + n - 1)^2  &\geq 0 \\
			-2 [ 2(\lambda_1 - 2b - a + n - 1) - 2] &\geq 0 \\
			\lambda_1 - 2b - a + n - 1  - 1  &\leq 0 \\
			\lambda_1 &\leq a + 2b + 2 - n
		\end{align*}
		For scalar $\lambda$ we immediately get $\lambda_1 \leq 2-n \leq a + 2b + 2 - n.$ In the spinorial case we get $\lambda_1 \leq 3/2 - n \leq a + 2b + 2 - n.$ In the remaining case we actually have $\lambda_2 \geq 1$ and so \eqref{eq:so_even_basic_dirac} implies $\lambda_1 \leq 1 + p -2n$ which is indeed less than or equal to $a + 2b - n.$
	\end{proof}

	\begin{thm}[General case]
		Let $\lambda$ be as in case 3 and let \eqref{eq:so_even_basic_dirac} holds. Then the Dirac inequality \eqref{eq:so_even_basic} holds for any Schmid module $s$. 
	\end{thm}
	\begin{proof}
		
		Thanks to the previous lemma we only have to prove that for $a \geq 1$ we have
		\[
		\| (\lambda - s_{a,0})^+ + \rho \|^2 - \| \lambda + \rho \|^2 \geq 0,
		\]
		with strict inequality if \eqref{eq:so_even_basic_dirac} is strict. It follows by induction from combining the following two inequalities
		\begin{gather*}
			\| (\lambda  - s_{a,0})^+ + \rho \|^2  \geq \| (\lambda' - s')^+ + \rho\|^2 \\
			\| \lambda' + \rho \|^2 \geq   \| \lambda + \rho \|^2
		\end{gather*}
		where
		\[
		\lambda' = (\lambda - s_1)^+  \quad s' = s_{a,0} - s_1 = s_{a-1, 0}.
		\]
		The first one immediately follows from the Lemma \ref{gen prv} for $\mu = (\lambda - s_1)^+$ and $\nu = (s_{a,0} - s_1)^+$ and the second one is the Dirac inequality for $s_1$ which is nothing but \eqref{eq:so_even_basic_dirac}. 
		
		What remains is to check that $\lambda'$ satisfies the Dirac inequality for $s' = s_{a-1,0}.$ It can happen that $\lambda'$ falls into the spinor or scalar case. For $\lambda = (\lambda_1, 1, 0, \ldots, 0)$ the inequality \eqref{eq:so_even_basic_dirac} takes the form $\lambda_1 + 1 \leq 2 + 2 -2n$ which means that $\lambda_1' = \lambda_1 - 1$ satisfies $\lambda_1' \leq 2 - 2n.$ Looking back at the Theorem \ref{scalar} we see that $\lambda'$ satisfies the Dirac inequality with any Schmid module. Analogously, for $\lambda = (\lambda_1, 3/2, 1/2, \ldots, \pm 1/2)$ we have $\lambda_1 + \lambda_2 = \lambda_1 + 3/2 \leq 2 + 2 - 2n$ which gives $\lambda_1' = \lambda_1 - 1 \leq  3/2 - 2n$ which is below the unitarizability bound $3/2 - n$ for the spinor case.

		In all other cases $\lambda'$ is of general type. If $p > 2$, then $ \lambda_1' + \lambda_2' = \lambda_1 - 1 + \lambda_2$ and using \eqref{eq:so_even_basic_dirac} and $p' = p - 1$ we see that this is less than or equal to $2 + p' - 2n.$ Hence the Dirac inequality is satisfied by the induction hypothesis. For $p = 2$ we have similarly $\lambda_1' + \lambda_2' = \lambda_1 + \lambda_2 - 2 \leq 2 + 2 - 2n  \leq 2 + p' - 2n$ since $p' \geq 2.$ 
	\end{proof}

	\subsection{Dirac inequality for $\frso(2, 2n-1)$}
	The basic Schmid $\frk$-submodules of $S(\frp^-)$ have lowest weights $-s_1$ or $-s_2$, where
	\[
	s_1=(1, 1, 0, \ldots 0), \
	s_2 = (2, 0, 0, \ldots 0).
	\] 
	
	Moreover, all irreducible $\frk$-submodules of $S(\frp^-)$ have lowest weights $-s_{a,b}$, where
	$s_{a,b} = (2b + a, a, 0, \ldots 0).$

	The highest weight $(\frg,K)$-modules have highest weights of the form
	$\lambda = (\lambda_1, \ldots, \lambda_n)$, where 
	\[
	\lambda_2 \geq \lambda_3 \geq \cdots \lambda_{n} \geq 0,
	\]
	
	$\lambda_i - \lambda_j \in \bbZ$ and $2\lambda_i \in \mathbb{N}_0$ for all $2\leq i, j \leq n.$

	In this case $\rho = (n-1/2, n-3/2, \ldots, 1/2)$.
	
	The basic necessary condition for unitarity is, as before, the Dirac inequality
	\eq
	\label{basic di so2, 2n-1}
	\|(\la-s_1)^++\rho\|^2\geq\|\la+\rho\|^2.
	\eeq

	The basic Dirac inequality for a Schmid module is
	\begin{equation}\label{eq:so_odd_basic}
		\| (\lambda - s)^+ + \rho \|^2 \geq \| \lambda + \rho \|^2
	\end{equation}
	This is equivalent to
	\begin{equation}
		2 \langle \gamma \,|\, \lambda + \rho \rangle \leq \|\gamma\|^2
	\end{equation}
	where $\gamma$ is defined by $(\lambda - s)^+ = \lambda - \gamma.$

	\begin{lem}
		The basic Dirac inequality for $s = s_1$ is given by
		\begin{equation}\label{eq:so_odd_basic_dirac}
			\begin{aligned}
				\lambda_1 &\leq 0 & \text{ for } \lambda=(\lambda_1, 0,\ldots, 0)  \\
				\lambda_1 &\leq 1-n & \text{ for } \lambda=(\lambda_1, 1/2,\ldots, 1/2) \\
				\lambda_1 + \lambda_2 &\leq 1 + p - 2n &\text{ for } \lambda=(\lambda_1, \lambda_2, \ldots, \lambda_p, \ldots, \lambda_n) \\
				&  & \text{ where } 1 \leq \lambda_2 = \cdots = \lambda_p > \lambda_{p+1} \text{ and } 2 \leq p \leq n.
			\end{aligned}
		\end{equation}
	\end{lem}
	\begin{proof}
		\textbf{Case 1:} $\lambda = (\lambda_1, 0, \ldots, 0)$
		
		In this case we have
		\begin{align*}
			(\lambda - s_1)^+ &= (\lambda_1 - 1, -1, 0, \ldots, 0)^+ \\
			&= (\lambda_1 - 1, 1, 0, \ldots, 0) \\
			&= \lambda - (\epsilon_1 - \epsilon_2)
		\end{align*}
		which shows that $\gamma = \epsilon_1 - \epsilon_2$ and \eqref{eq:so_odd_basic} reduces to 
		\(
		\lambda_1 + n - 1/2 -  (n - 3/2) \leq 1
		\)
		which is equivalent to $\lambda_1 \leq 0.$
		
		\medskip
		
		\textbf{Case 2:} $\lambda = (\lambda_1, 1/2, \ldots, 1/2)$
		
		In this case we have 
		\begin{align*}
			(\lambda - s_1)^+ &= (\lambda_1 - 1, -1/2, 1/2, \ldots,1/2, 1/2)^+ \\
			&= (\lambda_1 - 1, 1/2, 1/2, \ldots,1/2, 1/2) \\
			&= \lambda - \epsilon_1 
		\end{align*}
		Plugging $\gamma = \epsilon_1$ into \eqref{eq:so_odd_basic} we obtain $2(\lambda_1 + n - 1/2) \leq 1$ which gives 
		\begin{equation}
			\lambda_1 \leq 1 - n.
		\end{equation}
		
		\textbf{Case  3:} $1 \leq \lambda_2 = \cdots = \lambda_p > \lambda_{p+1}$
		
		The basic Dirac inequality \eqref{eq:so_odd_basic} for $s=s_1$ has $\gamma = \epsilon_1 + \epsilon_p$ and
		\begin{align*}
			2(\lambda_1 + n-1/2 + \lambda_p + n - p + 1/2) &\leq 2 \\
			\lambda_1 + \lambda_2 &\leq 1 + p - 2n
		\end{align*}
		
	\end{proof}
	
	\begin{thm}[Scalar case]
		Let $\lambda = (\lambda_1, 0, \ldots, 0)$ such that $\lambda_1 < 3/2-n$. Then \eqref{eq:so_odd_basic} holds for any Schmid module $s$.
	\end{thm}
	\begin{proof}
		The Dirac inequality for the second basic Schmid  yields $\lambda_1 \leq 3/2 -n.$  
		
		General Schmid module has highest weight $s_{a,b} = (2b + a, a, 0, \ldots 0)$ and similarly as before we get that the Dirac inequality is equivalent to
		\[
		(2b+a)(\lambda_1 + n + 1/2 -1) - a(n+1/2-2)  \leq 2b^2 + 2ab + a^2.
		\]
		Using $\lambda_1 \leq 3/2 - n$  we see that it is sufficient to prove that 
		\begin{align*}
			(2b+a)(3/2 - n + n +1/2 -1) - a(n+1/2-2)  &\leq 2b^2 + 2ab + a^2 \\
			2b + a + (3/2-n)a &\leq 2b^2 + 2ab + a^2 \\
			0 &\leq 2b(b-1) + 2ab + a(a + n - 5/2).
		\end{align*}
		Since $a, b \geq 0$ and $n \geq 3$ we are done.
	\end{proof}
	
	\begin{lem}
		Let $s_{a,b} = (2b + a, a, 0,\ldots, 0)$ be a Schmid module with $a \geq 1$ and $b \geq 0.$ Let $\lambda$ satisfy \eqref{eq:so_odd_basic_dirac} in the case $\lambda_2 \neq 0$ and $\lambda_1 \leq 3/2-n$ in the scalar case. Then we have
		\begin{equation}\label{eq:so_odd_induction_b}
			\| (\lambda - s_{a,b+1})^+ + \rho \|^2 - \| \lambda + \rho \|^2 \geq \| (\lambda - s_{a, b})^+ + \rho\|^2 - \| \lambda + \rho \|^2.
		\end{equation}
		%and the inequality is strict if the inequality from the assumption is strict. 
	\end{lem}
	\begin{proof}
		Since the weights differ only in the first coordinate, the difference of the left hand side and the right hand side is just the difference of squares on the first coordinate:
		\begin{align*}
			(\lambda_1 - 2b - 2 - a + n - 1/2)^2 - (\lambda_1 - 2b - a + n - 1/2)^2  &\geq 0 \\
			-2 [ 2(\lambda_1 - 2b - a + n - 1/2) - 2] &\geq 0 \\
			\lambda_1 - 2b - a + n - 1/2  - 1  &\leq 0 \\
			\lambda_1 &\leq a + 2b + 3/2 - n
		\end{align*}
		For scalar $\lambda$ we immediately get $\lambda_1 \leq 3/2-n \leq a + 2b + 3/2 - n.$ In the spinorial case we get $\lambda_1 \leq 1 - n \leq a + 2b + 3/2 - n.$ In the remaining case we actually have $\lambda_2 \geq 1$ and so \eqref{eq:so_odd_basic_dirac} implies $\lambda_1 \leq p - 2n$ which is indeed less than or equal to $a + 2b  - n.$
	\end{proof}

	\begin{thm}
		Let $\lambda$ be as in case 2 or as in case 3 and let \eqref{eq:so_odd_basic_dirac} holds. Then the Dirac inequality \eqref{eq:so_odd_basic} holds for any Schmid module $s$.
	\end{thm}
	\begin{proof}
		
		Thanks to the previous lemma we only have to prove that for $a \geq 1$ we have
		\[
		\| (\lambda - s_{a,0})^+ + \rho \|^2 - \| \lambda + \rho \|^2 \geq 0,
		\]
		with strict inequality if \eqref{eq:so_odd_basic_dirac} is strict. It follows by induction from combining the following two inequalities
		\begin{gather*}
			\| (\lambda  - s_{a,0})^+ + \rho \|^2  \geq \| (\lambda' - s')^+ + \rho\|^2 \\
			\| \lambda' + \rho \|^2 \geq   \| \lambda + \rho \|^2
		\end{gather*}
		where
		\[
		\lambda' = (\lambda - s_1)^+  \quad s' = s_{a,0} - s_1 = s_{a-1, 0}.
		\]
		The first one immediately follows from the Lemma \ref{gen prv} for $\mu = (\lambda - s_1)^+$ and $\nu = (s_{a,0} - s_1)^+$ and the second one is the Dirac inequality for $s_1.$
		
		If $\lambda$ is in the spinor case, then $\lambda'$ is also in the spinor case. In the general situation $\lambda'$ can fall into all three cases. For $\lambda = (\lambda_1, 3/2, 1/2, \ldots, 1/2)$ we have spinorial $\lambda'$ with $\lambda_1' = \lambda_1 - 1$ and since our $\lambda$ satisfies \eqref{eq:so_odd_basic_dirac} we have $\lambda_1 \leq 3/2 - 2n$ which means that $\lambda_1' \leq 1/2 - 2n.$ For $\lambda'$ of general type the exactly same reasoning as in the even case ($\mathfrak{so}(2, 2n-2)$) finishes the proof.
	\end{proof}

\end{document}